\documentclass[12pt]{article}
\usepackage[pagewise]{lineno}
\usepackage[utf8]{inputenc}
\usepackage{amsfonts}
\usepackage{amsmath,amsthm}
\usepackage{authblk}
\usepackage{amssymb}
\usepackage{array}
\usepackage{bm}
\usepackage{caption}
\usepackage{enumerate}
\usepackage{epsfig}
\usepackage{epstopdf}
\usepackage{esvect}
\usepackage{float}
\usepackage{graphics}
\usepackage{graphicx}
\usepackage{latexsym}
\usepackage{listings}
\usepackage{makeidx}
\usepackage{mathtools}
\usepackage{rotating}
\usepackage{tabularx}
\usepackage{subfig}
\usepackage{fixltx2e}
\usepackage{pdflscape}
\usepackage{stackengine}

\usepackage{pgf, tikz}
\usetikzlibrary{arrows, automata}
\usetikzlibrary{positioning}
\graphicspath{ {./images/} }

\graphicspath{ {images/} }
\tikzset{inner sep=0pt}

\newtheorem{observation}{Observation}

\newtheorem{theorem}{Theorem}
\newtheorem{remark}{Remark}
\newtheorem{lemma}{Lemma}
\newtheorem{definition}{Definition}
\newtheorem{corollary}{Corollary} 

\newcommand{\floor}[1]{\left\lfloor #1\right \rfloor}
\newcommand{\ceil}[1]{\left\lceil #1 \right\rceil}

\title{Maximum Rectilinear Crossing Number of Uniform Hypergraphs}

\author[1]{Rahul Gangopadhyay\thanks{rahulg@iiitd.ac.in}}
\author[2]{Ayan\thanks{saifk@iiitd.ac.in}}
\affil[1]{Laboratory of Combinatorial and Geometric Structures,\\Moscow Institute of Physics and Technology, Dolgoprudny, Moscow Region, 141700, Russian Federation.}
\affil[2]{ IIIT Delhi, New Delhi, India.}
\begin{document}
\maketitle

\begin{abstract}

We improve the lower bound on the $d$-dimensional rectilinear
crossing number of the complete $d$-uniform hypergraph having
$2d$ vertices  to $\Omega\left(\dfrac{(4\sqrt{2}/3^{3/4})^d}{d}\right)$ from $\Omega(2^d \sqrt{d})$. We also establish that the $3$-dimensional rectilinear
crossing number of a complete $3$-uniform hypergraph having
$n \geq 9$ vertices is at least $\dfrac{43}{42}\dbinom{n}{6}$.\\

We prove that the maximum number of crossing pairs of hyperedges in a $4$-dimensional
rectilinear drawing of the complete $4$-uniform hypergraph having $n$ vertices is $13\dbinom{n}{8}$.
We also prove that among all 
$4$-dimensional rectilinear drawings of a complete
$4$-uniform hypergraph having $n$ vertices,
the number of crossing pairs of hyperedges is maximized if
all its vertices are placed at the vertices of a $4$-dimensional
neighborly polytope.
Our result proves the conjecture by Anshu et al. [Anshu, Gangopadhyay, Shannigrahi, and Vusirikala, 2017] for $d=4$.

We  prove that the maximum $d$-dimensional rectilinear
crossing number of a complete $d$-partite $d$-uniform balanced
hypergraph is $(2^{d-1}-1){\dbinom{n}{2}}^d$. \par
We then prove that finding the maximum $d$-dimensional
rectilinear crossing number of an arbitrary $d$-uniform
hypergraph is NP-hard.

We give a randomized scheme to create
a $d$-dimensional rectilinear drawing of a $d$-uniform hypergraph
$H$ such that, in expectation the total number of crossing pairs of hyperedges 
is a constant fraction of the maximum $d$-dimensional rectilinear
crossing number of $H$.

\end{abstract}

\section{Introduction}
In a topological drawing of a graph $G$ in $\mathbb{R}^2$, vertices are mapped as points in $\mathbb{R}^2$ and edges are drawn as simple Jordan curves, so that any pair of edges cross each other at most once, and adjacent edges do not cross each other. In a topological drawing of a graph, no edge passes through a vertex apart from its endpoints and no two edges touch each other (except crossing). Also, no three edges pass through a common crossing point in a topological drawing of a graph. A rectilinear drawing of a graph in $\mathbb{R}^2$ represents its vertices as points in general position, i.e., no three points are colinear, and its edges are straight line segments connecting the corresponding vertices. 
In a rectilinear  drawing of a graph, a pair of edges are said to be \emph{crossing} if they are vertex-disjoint and contain a common point in their relative interiors.
The rectilinear crossing number of a graph $G$, denoted by $\overline {cr}(G)$, is the minimum number of crossing pairs of edges among all rectilinear drawings of the graph. 
A convex drawing of a graph $G$ is a rectilinear drawing with its vertices in  convex position in $\mathbb{R}^2$. 
There are other variants of graph crossing numbers which are comprehensively discussed in~\cite{SCH}. 
Most of the problems on crossing numbers deal with the minimization of the number of crossings among all possible drawings of the graph.\par

The maximum crossing number of a graph $G$, denoted by \emph{max-}$cr(G)$, is the maximum number of crossing pairs of edges among all  topological drawings of $G$ in which no three distinct edges cross in one point and any pair of edges share a common point if and only if they are adjacent or they cross each other, see~\cite{SCH}. Ringel~\cite{RIN} introduced the maximum rectilinear crossing number for a graph $G$, being the maximum number of crossing pairs of edges among all rectilinear drawings of $G$. Let us denote the maximum rectilinear crossing number for a graph $G$ by \emph{max-}$\overline {cr}(G)$. 
Verbitsky~\cite{VER} gave an approximation algorithm, which in expectation provides a $1/3$ approximation guarantee on the maximum rectilinear crossing number problem.
The study also showed that the maximum rectilinear crossing number of a planar graph having $n$ vertices is less than $3n^2$.
Bald et al.~\cite{BLD} de-randomized Verbitsky's algorithm and showed that it is NP-hard to find the maximum rectilinear crossing number of an arbitrary graph. Chimani et al.~\cite{CHI} showed that the maximum rectilinear crossing number of an arbitrary graph can not be approximated better than \textit{Max-Cut} problem thus proving it to be APX-hard. In the same paper~\cite{CHI}, they proved that  finding \emph{max-}$cr(G)$ is NP-complete for an arbitrary graph $G$. \par

A hypergraph, a natural generalization of a graph, is defined as an ordered pair $(V, E)$ where $V$ is the set of vertices and $E \subseteq 2^V \setminus \{\emptyset\}$ is the set of hyperedges.
A hypergraph is said to be $d$-uniform if each hyperedge contains exactly $d$ vertices.
Let $K_n^d$ denote the \emph{complete $d$-uniform hypergraph} having $n$ vertices and $\dbinom{n}{d}$ hyperedges.
We can partition the vertex set of a \emph{$d$-uniform $d$-partite} hypergraph into $d$ disjoint parts such that each of the $d$ vertices in each hyperedge belongs to a different part, and it is balanced if each of the parts has the same number of vertices.
A balanced $d$-uniform $d$-partite hypergraph having $n$ vertices in each part is complete if it has all $n^d$ hyperedges and it is denoted by $K_{d\times n}^d$.
Dey and Pach~\cite{DP} extended the idea of a rectilinear drawing of a graph to a rectilinear drawing of a hypergraph. 
Consider a set $P$ having $n \geq d+1$ points in $\mathbb{R}^d$. 
The points in $P$ are said to be in general position if no set of $d+1$ points of $P$ lie on a $(d-1)$-dimensional hyperplane.  Let us  denote  the convex hull of  $P$ by  $Conv(P)$.

In a \emph{$d$-dimensional rectilinear drawing} of a $d$-uniform hypergraph $H$, the vertices of $H$ are placed in general position in $\mathbb{R}^d$ and the hyperedges are drawn as the convex hull of $d$ corresponding vertices, i.e. $(d-1)$-simplices~\cite{AS}.
In a $d$-dimensional rectilinear drawing of $H$, a pair of hyperedges are said to \emph{cross} each other if they are vertex-disjoint and contain a common point in their relative interiors~\cite{AS,DP}.
The \emph{$d$-dimensional rectilinear crossing number} of $H$, denoted by $\overline {cr}_d(H)$, is the minimum number of crossing pairs of hyperedges among all $d$-dimensional rectilinear drawings of $H$~\cite{AS}.
Dey and Pach~\cite{DP} proved that $H$ can have at most $O(n^{d-1})$ hyperedges if $\overline {cr}_d(H)=0$. The first non-trivial lower bound of $\Omega(2^d \log d / \sqrt{d})$ on $\overline {cr}_d(K_{2d}^d)$ was proved by Anshu and Shannigrahi~\cite{AS}.    Anshu et al.~\cite{AGS} proved that $\overline {cr}_d(K_{2d}^d)= \Omega(2^d)$ with the bound being later improved to $\Omega(2^d \sqrt{d})$~\cite{GS}. Gangopadhyay et al.~\cite{GS1} proved that $\overline {cr}_d\left(K_{d \times n}^d\right)= \Omega(2^d)(n/2)^d ((n-1)/2)^d$ for $n \geq 3$ and sufficiently large $d$. \par

The points in $P$ are in convex position if none of these points can be expressed as a convex combination of the remaining points in $P$. 
A \emph{$d$-dimensional convex drawing} of a $d$-uniform hypergraph $H$ is a $d$-dimensional rectilinear drawing of it where all its vertices are in convex position as well as in general position in $\mathbb{R}^d$.
In this paper, we define the \emph{maximum $d$-dimensional rectilinear crossing number} of a $d$-uniform hypergraph $H$, denoted by \emph{max-}$\overline {cr}_d(H)$, as the maximum number of crossing pairs of hyperedges among all $d$-dimensional rectilinear drawings of $H$.\par

Consider a set $P$ having $n\geq d+1$ points in convex position in $\mathbb{R}^d$.
Let us assume that the affine hull of the points in $P$ is the entire space $\mathbb{R}^d$. The convex hull of $P$, denoted by
$Conv(P)$, is a $d$-dimensional convex polytope, and the points in $P$ are the vertices of $Conv(P)$. 
A $d$-dimensional convex polytope is \emph{$k$-neighborly} if any subset of its vertex set containing at most $k$ vertices spans a face.
A $d$-dimensional convex polytope can be at most $\floor{d/2}$-neighborly unless it is a $d$-simplex. 
A $d$-dimensional $\floor{d/2}$-neighborly polytope is called $d$-dimensional \emph{neighborly polytope}. 
The \emph{$d$-dimensional moment curve} $\gamma_d$ is defined as~$\gamma_d= \{(t,t^2, \ldots, t^d): t \in~\mathbb{R}\}$.
Let $p_i=(t_i,t_i^2, \ldots, t_{i}^d)$ and $p_j=(t_j,t_j^2, \ldots, t_{j}^d)$ be two points on $\gamma_d$.
We say that the point $p_i$ precedes the point $p_j$ ($p_i \prec p_j$) if $t_i < t_j$. The $d$-dimensional cyclic polytope is an example of a $d$-dimensional neighborly polytope where all of its vertices are placed on $\gamma_d$.\par

Anshu et al.~\cite{AGS} proved that placing all the vertices of  $K_{2d}^d$ at the vertices of a $d$-dimensional cyclic polytope gives rise to a particular $d$-dimensional rectilinear drawing of $K_{2d}^d$ having $c_d^m$ crossing pairs of hyperedges, where $c_d^m$ is defined as follows:\\

\[c_d^m= \begin{cases} 
   
\dbinom{2d-1}{d-1}-\sum\limits_{i=1}^{d/2}\dbinom{d}{i}\dbinom{
d-1} {i-1} & \textrm{if~$d$~is~even},\\
\\
\dbinom{2d-1}{d-1}-1-\sum\limits_{i=1}^{\floor{d/2}}\dbinom{d-1}{i}\dbinom{
d}{i} & \textrm{if~$d$~is~odd}.\\ 
  \end{cases}\\
\]

In~\cite{AGS}, it was conjectured that the maximum number of crossing pairs of hyperedges in a $d$-dimensional convex drawing of $K_{2d}^d$ is $c_d^m$ for each $d \geq 2$.
It is trivially true for $d = 2$, since any convex drawing of the complete graph $K_n$ produces $\dbinom{n}{4}$ pairs of crossing edges. 
In~\cite{AGS}, the authors also proved that a $3$-dimensional rectilinear drawing of $K_{6}^3$ can have at most three crossing pairs of hyperedges, implying that $K_{n}^3$ can have at most $3\dbinom{n}{6}$ crossing pairs of hyperedges in any of its $3$-dimensional rectilinear drawings.
They also showed that any $3$-dimensional convex drawing of $K_{6}^3$ has three crossing pairs of hyperedges.\par

Note that we need at least $2d$ vertices to form a crossing pair of hyperedges since they need to be vertex-disjoint, and each set of $2d$ vertices creates distinct crossing pairs of hyperedges. 
If placing the vertices of $K_{2d}^d$ on $\gamma_d$ maximizes the number of crossing pairs of hyperedges among all  its $d$-dimensional rectilinear drawings, then \emph{max-}$\overline {cr}_d(K_{n}^d)= c_d^m \dbinom{n}{2d}$ since every set of $2d$ vertices on $\gamma_d$ spans $c_d^m$ crossing pairs of hyperedges.
Using these definitions and notations, we describe our contributions in this paper.\par

\subsection{Our Contributions}

In Section~\ref{sec1}, we  improve the lower bound on $\overline {cr}_3(K_{n}^3)$ for $n \geq 9$.  In particular, we prove the following results.

\begin{theorem}
\label{thm00}
For $n \geq 9$, $\overline {cr}_3(K_{n}^3) \geq \dfrac{43}{42}\dbinom{n}{6}$.
\end{theorem}

In Section~\ref{sec11}, we improve the lower bound on $\overline {cr}_d(K_{2d}^d)$ to $\Omega\left(\dfrac{(4\sqrt{2}/3^{3/4})^d}{d}\right)$.

\begin{theorem}
\label{thm0}
For sufficiently large $d$, $\overline {cr}_d(K_{2d}^d)= \Omega\left(\dfrac{(4\sqrt{2}/3^{3/4})^d}{d}\right)$ which is approximately $\Omega\left(\dfrac{2.481^d}{d}\right)$.
\end{theorem}

In Section~\ref{sec3}, we prove the conjecture proposed by Anshu et al.~\cite{AGS} for $d=4$.

\begin{theorem}
\label{thm1}
For $n \ge 8$, \emph{max-}$\overline {cr}_4(K_{n}^4)= 13 \dbinom{n}{8}$.
\end{theorem}

It is natural to ask the maximum $d$-dimensional rectilinear crossing number of $K^d_{d \times n}$.
In Section~\ref{sec4}, we prove the following result.

\begin{theorem}
\label{thm2}
For $n \ge 2$ and $d \ge 2$, \emph{max-}$\overline {cr}_d(K^d_{d \times n})= (2^{d-1}-1) {\dbinom{n}{2}}^d$.
\end{theorem}

In Section~\ref{sec5}, we prove that for $d \geq 3$, finding the maximum $d$-dimensional rectilinear crossing number of an arbitrary $d$-uniform hypergraph is NP-hard.
Since this problem is NP-hard, we propose a randomized approximation algorithm, which in expectation gives a constant $\tilde{c_d}$ approximation guarantee on the maximum $d$-dimensional rectilinear crossing number problem. 
The constant $\tilde{c_d}$ is dependent on $d$. In particular, we prove the following results.

\begin{theorem}
\label{thm3}
 For $d \geq 3$, finding the maximum $d$-dimensional rectilinear crossing number of an arbitrary $d$-uniform hypergraph is NP-hard.
\end{theorem}

\begin{theorem}
\label{lemrandom}
Let $H=(V,E)$ be a $d$-uniform hypergraph for $d \geq 3$. Let $F$ be the total number of pairs of vertex-disjoint hyperedges. There exists a $d$-dimensional rectilinear drawing $D$ of $H$ such that there are at least $\tilde{c_d} \cdot F$ crossing pairs of hyperedges in $D$, where $\tilde{c_d}$ is a constant.
\end{theorem}

\section{Improved Lower Bound on $\overline {cr}_3(K_{n}^3)$}\label{sec1}

In this section, we improve the lower bound on the $3$-dimensional rectilinear crossing number of $K_n^3$ for $n\geq 3$.  We state two lemmas that are used to improve the lower bound on the $3$-dimensional rectilinear crossing number of $K_n^3$ for $n \geq 9$.

\begin{lemma}~\cite{BISZ} \label{Erdos-szekers} 
Every set of nine points in general position in $\mathbb{R}^3$ contains a subset of six points that are in convex position. 
\end{lemma}
Note that Lemma~\ref{Erdos-szekers} is the $3$-dimensional analogue of the Erd\H{o}s-Szekeres theorem. 

\begin{lemma}~\cite{AGS} \label{Conv3d}
The number of crossing pairs of hyperedges in a $3$-dimensional rectilinear drawing of $K_6^3$ is at least one.
The number of crossing pairs of hyperedges in a $3$-dimensional rectilinear drawing of $K_6^3$ is three if its vertices are in convex as well as general position in $\mathbb{R}^3$.
\end{lemma}

\noindent\textbf{Proof of Theorem~\ref{thm00}.} Let $V=\{v_1,v_2, \ldots, v_{9}\}$ denote the set of nine points corresponding to the vertices of $K_{9}^3$ in any of its $3$-dimensional rectilinear drawings.
Lemma~\ref{Erdos-szekers} implies that in such a $3$-dimensional rectilinear drawing of $K_{9}^3$ there exist six points which are in general as well as convex position in $\mathbb{R}^3$. 
Let us consider the sub-hypergraph $H$ of $K_{9}^3$ induced by the six vertices corresponding to these points. 
Note that $H$ is isomorphic to $K_{6}^3$. 
Lemma~\ref{Conv3d} implies that $H$ contains three crossing pairs of hyperedges. 
Also, note that there are $\dbinom{9}{6}$ distinct sub-hypergraphs of $K_{9}^3$ which are isomorphic to $K_{6}^3$.
Lemma~\ref{Conv3d} also implies that each of these $\dbinom{9}{6}$ distinct sub-hypergraphs contains at least one crossing pair of hyperedges, and one of them, i.e., $H$ contains three crossing pairs of hyperedges. 
Also, note that the crossing pairs of hyperedges spanned by any  set of six vertices are distinct from the crossing pairs of hyperedges spanned by any other set of six vertices. 
The total number of crossing pairs of hyperedges in a $3$-dimensional rectilinear drawing of $K_{9}^3$ is at least $\dbinom{9}{6}-1+3=86$. 
This implies that $\overline {cr}_3(K_{9}^3) \geq 86$. \par

Consider a $3$-dimensional rectilinear drawing of $K_{n}^3$ where $n\geq9$. 
Note that $K_n^3$ contains $\dbinom{n}{9}$ distinct induced sub-hypergraphs, each of which is isomorphic to $K_9^3$. 
Also, each crossing pair of hyperedges is contained in $\dbinom{n-6}{3}$ distinct induced sub-hypergraphs which are isomorphic to $K_9^3$. 
Using these two facts, we obtain that $\overline{cr}_3\left(K_{n}^3\right) \geq 86{\dbinom{n}{9}}\bigg /\dbinom{n-6}{3} = \dfrac{43}{42}\dbinom{n}{6}$. 
\begin{remark}

The currently best known lower and upper bounds on the rectilinear crossing number of a complete graph with $n$ vertices are $0.37997\dbinom{n}{4}+\Theta(n^3)$ and $0.3804491869 \dbinom{n}{4}+ \Theta(n^3)$, respectively~\cite{Alb,FML}.
It will be an interesting challenge to reduce the gap between the lower bound and the upper bound on $\overline {cr}_3(K_{n}^3)$ and make them very close to each other as it has been done in the case of the complete graph with $n$ vertices.\par

The exact values of rectilinear crossing number of complete graphs having $n$ vertices are known when $n \leq 27$~\cite{ABRGO}.
Similarly, determining the exact values of $\overline {cr}_3(K_{n}^3)$ for small values of $n$ can be an interesting direction of research. Also, determining the exact values of $\overline {cr}_3(K_{n}^3)$ for small values of $n$ can improve the general lower bound on $\overline {cr}_3(K_{n}^3)$. 

\end{remark}

\section{Gale Transform and Gale Diagram}\label{sec2}

We use the properties of the Gale transform~\cite{GL} and the affine Gale diagram to prove Theorems~\ref{thm0} and~\ref{thm1}.
In this section, we describe how to obtain a Gale transform and an affine Gale diagram of a  sequence of points and discuss their properties.  \par

Let $A=\langle a_1,a_2, \ldots, a_n \rangle$ be a sequence of $n$ points in $\mathbb{R}^d$ such that their affine hull is $\mathbb{R}^d$. 
The Gale transform of $A$, denoted by $D(A)$, is a sequence of $n$ vectors $\langle g_1,g_2, \ldots,g_n \rangle$ in $\mathbb{R}^{n-d-1}$. 
Let the coordinate of $a_i$ be $(x_1^i,x_2^i, \ldots, x_d^i)$.
Let us consider the following matrix $M(A)$:

\begin{center}
$
M(A) = 
\begin{bmatrix}
x_1^1&x_1^2&\cdots & x_1^n \\
x_2^1&x_2^2&\cdots & x_2^n \\
\vdots & \vdots & \vdots & \vdots\\ 
x_d^1&x_d^2&\cdots & x_d^n \\
1 & 1 & \cdots & 1
\end{bmatrix}
$.
\end{center}

Since at least $d+1$ points of $A$ are affinely independent, the dimension of the null space of $M(A)$ is $n-d-1$. 
Let $\{(b_1^1, b_2^1, \ldots, b_n^1), (b_1^2, b_2^2,$ $\ldots, b_n^2), \ldots, (b_1^{n-d-1},$ $b_2^{n-d-1}, \ldots, b_n^{n-d-1}) \}$ be a basis of it.
The vector $g_i$ in the sequence $D(A)$ of $n$ vectors is $g_i=(b_i^1$, $b_i^2$, $\ldots, b_i^{n-d-1})$.\par

A \emph{linear separation} of vectors in $D(A)$ is a partition of the vectors into $D^+(A)$ and $D^-(A)$ by a hyperplane passing through the origin. 
The opposite open half-spaces of the partitioning hyperplane contain the sets $D^+(A)$ and $D^-(A)$.
When $|D(A)|$ is even, a linear separation is called \emph{proper} if $|D^+(A)|= |D^-(A)|= |D(A)|/2$. 
Next, we state some interesting properties of the Gale transform of $A$. 

\begin{lemma}~\cite{JM}
\label{lem:conditiongale}
A sequence $\langle g_1,g_2, \ldots,g_n \rangle$ of $n$ vectors in $\mathbb{R}^{n-d-1}$ is a Gale transform of some  sequence of $n$ points in $\mathbb{R}^d$ if and only if the vectors in $\langle g_1,g_2, \ldots,g_n \rangle$ span $\mathbb{R}^{n-d-1}$ and $\sum_{i=1}^n g_i = \vv{\bm{0}}$. 
\end{lemma}

\begin{definition}[Totally cyclic vector  configuration:]  A vector configuration $V=\{v_1,v_2, \ldots, v_n\}$  $\subset \mathbb{R}^d$ is said to be  \textit{totally cyclic}, if there exists a vector $\beta=(\beta_1, \beta_2, \ldots, \beta_n)$ in $\mathbb{R}^n$ such that each $\beta_i >0$ and $\beta_1v_1+ \beta_2v_2+ \ldots + \beta_nv_n= \vec{0}$. 
\end{definition}

Lemma~\ref{lem:conditiongale} implies that the Gale transform $D(A)$ of $A$ is a totally cyclic vector configuration, and there is a positive dependence among the vectors of $D(A)$. 
This also implies that there does not exist a hyperplane, passing through the origin, such that all the vectors of $D(A)$ lie on one side of the hyperplane ~\cite{ZIE}. 
Note that any totally cyclic vector configuration of $n$ vectors in $\mathbb{R}^{n-d-1}$ that span $\mathbb{R}^{n-d-1}$ can serve as a Gale transform of some point sequence having $n$ points in $\mathbb{R}^d$ after proper scaling.

\begin{lemma}~\cite{JM}
\label{lem:generalpos}
Every set of $n-d-1$ vectors of $D(A)$ span $\mathbb{R}^{n-d-1}$ if and only if the points in $A$ are in general position in $\mathbb{R}^{d}$.
\end{lemma}

\begin{lemma}~\cite{JM}
\label{bjection}
Consider a tuple $(i_1,i_2, \ldots, i_k)$, where $1 \leq i_1 < i_2< \ldots < i_k \leq n$.
The convex hull of $\{a_{i_1},a_{i_2}, \ldots, a_{i_k}\}$ crosses the convex hull of $A \setminus \{a_{i_1},a_{i_2}, \ldots, a_{i_k}\}$ if and only if there exists a linear separation of the vectors in $D(A)$ into $\{g_{i_1},g_{i_2}, \ldots, g_{i_k}\}$ and $D(A) \setminus \{g_{i_1},g_{i_2}, \ldots, g_{i_k}\}$. 
\end{lemma}

\begin{lemma}~\cite{JM}
\label{lem:facelemma}
For $t \leq d$, consider a tuple $(i_1,i_2, \ldots, i_t)$, where $1 \leq i_1 < i_2< \ldots < i_t \leq n$. A $t$-element subset $A'= \{a_{i_1},a_{i_2}, \ldots, a_{i_t}\} \subset A$ forms a $(t-1)$-dimensional face of $Conv(A)$ if and only if the relative interior of the convex hull of the points in $D(A)\setminus \{g_{i_1},g_{i_2}, \ldots, g_{i_t}\}$ contains the origin.
\end{lemma}

\begin{lemma}~\cite{GRU}\label{neigh}
Let the points in $A$ be in general as well as convex position in $\mathbb{R}^d$. A $d$-dimensional polytope formed by the convex hull of the points in $A$ is $t$-neighborly if and only if each of the linear separations of $D(A)$ contains at least $t+1$ vectors in each of the open half-spaces created by the corresponding linear hyperplane.
\end{lemma}

We consider the points in $A$ to be in general position.
Let $D(A)$ be a Gale transform of $A$.
Due to the general position of the points in $A$, Lemma~\ref{lem:generalpos} implies that none of the vectors in $D(A)$ is equal to the zero vector, i.e., $\forall i~g_i \neq \vec{0}$.

Consider a hyperplane $\bar{h}$ that is not parallel to any vector in $D(A)$ and not passing through the origin. We obtain an affine Gale diagram of $A$, denoted by $\overline{D(A)}$,  by the following way. 

For each $1\leq i \leq n$, we extend the vector $g_i \in D(A)$ either in the direction of $g_i$ or in its opposite direction until it cuts $\bar{h}$ at the point $\overline {g_i}$. 
We color $\overline {g_i}$ as \emph{red} (denoted as triangles in Figure~\ref{fig:311}) if the projection is in the direction of $g_i$, and \emph{violet} (denoted as squares in Figure~\ref{fig:311}) otherwise. 
The affine Gale diagram  $\overline{D(A)}$ of $A$ is the sequence of $n$ points $\langle \overline {g_1}, \overline {g_2}, \ldots, \overline {g_n} \rangle$ in $\mathbb{R}^{n-d-2}$ along with their respective colors.\par

We define a {\emph{separation}} of the points in $\overline{D(A)}$ as a partition of the points in $\overline{D(A)}$ in two disjoint sets of points $\overline{D^+(A)}$ and $\overline{D^-(A)}$ contained in the opposite open half-spaces created by a hyperplane.\\

\begin{figure}[h!]
\begin{center}
\includegraphics[scale = 0.17]{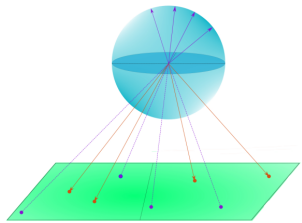}\\
\end{center}
\caption{An affine Gale diagram of 8 points in $\mathbb{R}^4$}
\label{fig:311}
\end{figure}

Let us define a \emph{balanced $2m$-partition} for a planar point set having an equal number of violet and red points in general position in $\mathbb{R}^2$. \\

\begin{definition}[Balanced $2m$-partition]
Let $T$ be a set of $n$ red and $n$ violet points in $\mathbb{R}^2$ such that all the $2n$ points are in general position.
A balanced $2m$-partition of $T$ is a partition of it into $\{X, T\setminus X\}$ such that the following properties hold.
\begin{itemize}
  \item The size of the set $X$ is $2m$.
  \item $X$ can be separated from $T\setminus X$ that contains the remaining $(2n-2m)$ points by a line not passing through any point of $T$.
  \item $X$ is balanced, i.e., it has an equal number of red and violet points.
\end{itemize}
 
\end{definition}

Since we are considering distinct balanced $2m$-partitions, i.e., the complementary pairs $\{X, T\setminus X\}$, we only consider them for $1 \leq m \leq \floor{n/2}$. 
We define a balanced $0$-partition to be a partition of $T$ into an empty set and $T$.
Note that there is only one balanced $0$-partition of a set.

\begin{definition}[$k$-partition]
Let $T$ be a set of $n$ red and $n$ violet points in $\mathbb{R}^2$ such that all the $2n$ points are in general position.
A $k$-partition of $T$ is a partition of it into $\{S, T\setminus S\}$ such that the size of the set $S \subseteq T$ is $k$ and $S$ can be separated from $T\setminus S$ by a line not passing through any point of $T$.
\end{definition}

\begin{definition}[Monochromatic $k$-partition]
Let $T$ be a set of $n$ red and $n$ violet points in $\mathbb{R}^2$ such that all the $2n$ points are in general position.
A monochromatic $k$-partition of $T$ is a partition of it into $\{Q, T\setminus Q\}$ such that the following properties hold.
\begin{itemize}
  \item The size of the set $Q \subseteq T$ is $k$.
  \item $Q$ can be separated from $T\setminus Q$ that contains  the remaining $(2n-k)$ points by a line not passing through any point of $T$.
  \item $Q$ is monochromatic, i.e., all the points in $Q$ are of the same color.
\end{itemize}
\end{definition}

\subsection{Properties of the Gale Diagram of Eight Points in $\mathbb{R}^4$ }

As already mentioned, in order to prove that \emph{max-}$\overline {cr}_4(K_{n}^4)= 13 \dbinom{n}{8}$, it is enough to show that placing the vertices of  $K_{8}^4$ on the $4$-dimensional moment curve maximizes the number of crossing pairs of hyperedges among all $4$-dimensional rectilinear drawings of it.\par

An affine Gale diagram maps a configuration of eight points in $\mathbb{R}^4$ to a configuration of eight points in a plane along with  colors associated with them such that some combinatorial information about the original point set is preserved.
We then analyse these planar point sets to prove the desired result. \par

Consider a set $A'$ of eight points in general position in $\mathbb{R}^4$.
Consider a Gale transform of $A'$, denoted by $D(A')$, which is a collection of eight vectors in $\mathbb{R}^3$.
Let us denote an affine Gale diagram of $A'$ by $\overline{D(A')}$.
We discuss a few properties of $\overline{D(A')}$ below.

\begin{observation}
\label{obs11}
There exists an affine Gale diagram $\overline{D(A')}$ of $A'$ having four red points and four violet points in $\mathbb{R}^2$ such that all the eight points are in general position.
\end{observation}

\begin{proof}
Consider a Gale transform $D(A')$ of $A'$, which is a set of eight vectors in $\mathbb{R}^3$. 
It is easy to note that there exists a $2$-dimensional hyperplane $\bar{h}$ passing through the origin that partitions $D(A')$ in two equal parts $D^+(A')$ and $D^-(A')$, each having four vectors.
Consider a hyperplane parallel to $\bar{h}$ and project the vectors in the above mentioned way to obtain an affine Gale diagram $\overline{D(A')}$ that has four violet points and four red points in $\mathbb{R}^2$, as shown in Figure~\ref{fig:311}. 
Also, note that the points in $\overline{D(A')}$ are in general position since no three of them are colinear.
If three points are colinear, it implies that the corresponding three vectors lie on a plane, which is a contradiction to Lemma~\ref{lem:generalpos} since the original points are in general position in $\mathbb{R}^4$. \qed
\end{proof}

\begin{observation}
\label{obs12}
Consider the affine Gale diagram $\overline{D(A')}$ having four red points and four violet points in $\mathbb{R}^2$ such that all the eight points are in general position.
The total number of proper linear separations (i.e., partition of eight vectors of $D(A')$ by a linear hyperplane into equal parts) in $D(A')$ is one more than the sum of the total number of balanced $2$-partitions of $\overline{D(A')}$ and the total number of balanced $4$-partitions of $\overline{D(A')}$.
\end{observation}

\begin{proof}

Consider any proper linear separation of vectors in $D(A')$ into $D^+(A')$ and $D^-(A')$. 
Note that this proper linear separation of vectors in $D(A')$ corresponds to a partition of points in $\overline{D(A')}$ into $\overline{D^+(A')}$ and $\overline{D^-(A')}$ by a line (this line is the intersection of the separating hyperplane with the hyperplane on which we projected the vectors to obtain the affine Gale Diagram).
Assume that there are $r$ red points and $b$ violet points in $\overline{D^+(A')}$.
This implies that there are $4-r$ red points and $4-b$ violet points in $\overline{D^-(A')}$. 
It is easy to note that the total number of vectors in $D^+(A')$ is equal to the number of red points in $\overline{D^+(A')}$ plus the number of violet points in $\overline{D^-(A')}$. 
This implies that $r+(4-b)$ is equal to $4$, and thus $r=b$.
This shows that each proper linear separation of vectors in $D(A')$ corresponds to a balanced $2m$-partition of $\overline{D(A')}$ for some $m$. 
Similarly, each balanced $2m$-partition of $\overline{D(A')}$ corresponds to a proper linear separation of vectors in $D(A')$.\par

The above argument shows that the total number of balanced $2$-partitions of $\overline{D(A')}$ plus the total number of balanced $4$-partitions of $\overline{D(A')}$ plus the balanced $0$-partition of $\overline{D(A')}$ is equal to the total number of proper linear separations in $D(A')$.
Note that we have not included a balanced $6$-partition since each balanced $6$-partition is the same as a balanced $2$-partition. 
Also, note that there is only one balanced $0$-partition of $\overline{D(A')}$. 
This balanced $0$-partition of $\overline{D(A')}$ corresponds to the proper linear separation of vectors in $D(A')$, which was used to obtain this affine Gale diagram.
This proves that total number of balanced $2$-partitions of $\overline{D(A')}$ plus the total number of balanced $4$-partitions of $\overline{D(A')}$ plus one is equal to the total number of proper linear separations in $D(A')$. \qed
\end{proof}

\begin{lemma}
\label{obs:13}
Consider the affine Gale diagram $\overline{D(A')}$ having four red points and four violet points in $\mathbb{R}^2$ such that all the eight points are in general position. Then, $\overline{D(A')}$ is an affine Gale diagram of a  $4$-dimensional neighborly polytope if and only if the following conditions hold.
\begin{itemize}
  \item Each $4$-partition in $\overline{D(A')}$ is a balanced $4$-partition.
  \item Each $2$-partition in $\overline{D(A')}$ is a balanced $2$-partition.
  \item There does not exist a monochromatic $3$-partition in $\overline{D(A')}$.
\end{itemize}
\end{lemma}
\begin{proof}

Lemma~\ref{neigh} implies that $D(A')$ is a Gale transform of a  $4$-dimensional neighborly polytope if and only if each of the linear separations of $D(A')$ contains at least three vectors in each of the open half-spaces created by the corresponding linear hyperplane.\par 

Consider any linear separation of vectors in $D(A')$ into $D^+(A')$ and $D^-(A')$.  Note that this linear separation of vectors in $D(A')$ corresponds to a partition of points in $\overline{D(A')}$ into $\overline{D^+(A')}$ and $\overline{D^-(A')}$ by a line (this line is the intersection of the separating hyperplane with the hyperplane on which we projected the vectors to obtain the affine Gale Diagram) and vice versa. \par 

It is easy to note that the total number of vectors in $D^+(A')$ is equal to the sum of the number of red points in $\overline{D^+(A')}$ and the number of violet points in $\overline{D^-(A')}$. 
Similarly, the total number of vectors in $D^-(A')$ is equal to the  sum of the number of red points in $\overline{D^-(A')}$ and the number of violet points in $\overline{D^+(A')}$. 

($\Rightarrow$) We first prove that if any of these three conditions mentioned above is violated, $D(A')$ is not a Gale transform of a  $4$-dimensional neighborly polytope having eight vertices.

\begin{description}

\item[Case 1.]

For the sake of contradiction, let us assume that there exists a $4$-partition in $\overline{D(A')}$ that is either monochromatic or contains three points of one color and one point of another color. 
Suppose it is monochromatic.
Then, this implies that there exists a linear hyperplane such that all the vectors of $D(A')$ lie in the same open half-space created by it, leading to a contradiction. 
Without loss of generality, let us assume that $\overline{D^+(A')}$ contains three points of one color and one point of the other color.
This implies that there exists a linear separation of $D(A')$ such that six vectors lie in the one side of the linear hyperplane and two vectors lie in the other side of the linear hyperplane.
Lemma~\ref{neigh} implies that $D(A')$ is not a Gale transform of a  $4$-dimensional neighborly polytope.
\item[Case 2.]
For the sake of contradiction, let us assume that the second condition is violated, i.e., there exists a monochromatic $2$-partition in $\overline{D(A')}$.
Without loss of generality, we assume that there exists a partition of points in $\overline{D(A')}$ into $\overline{D^+(A')}$ and $\overline{D^-(A')}$ by a line such that $\overline{D^+(A')}$ contains two points and both the points in $\overline{D^+(A')}$ are of the same color.
This implies that there exists a linear separation of $D(A')$ such that six vectors lie in the one side of the linear hyperplane and two vectors lie in the other side of the linear hyperplane, leading to a contradiction.

\item[Case 3.] 
For the sake of contradiction, we assume that there exists a monochromatic $3$-partition in $\overline{D(A')}$.
Without loss of generality, let us assume that $\overline{D^+(A')}$  contains three points having the same color.
This implies that there exists a linear separation of $D(A')$ such that seven vectors lie in the one side of the linear hyperplane and one vector lies in the other side of the linear hyperplane, leading to a contradiction.
\end{description}

($\Leftarrow$) In the following, we prove that if none of these three conditions is violated, any linear separation of $D(A')$ contains at least three vectors in each of the open half-spaces created by the corresponding linear hyperplane.
This implies that $D(A')$ is a Gale transform of a  $4$-dimensional neighborly polytope having eight vertices.

Note that for each linear separation of vectors in $D(A')$, there exists a partition of points in $\overline{D(A')}$ into $\overline{D^+(A')}$ and $\overline{D^-(A')}$.

Let us assume that $|\overline{D^+(A')}|=$ $|\overline{D^-(A')}|=4$.
Since each $4$-partition in $\overline{D(A')}$ is a balanced $4$-partition, any such partition corresponds to a proper linear separation of $D(A')$, as shown in the proof of Observation~\ref{obs12}.
Note that in any proper linear separation of $D(A')$ each open half-space contains four vectors of $D(A')$.

Let us assume that $|\overline{D^+(A')}|= 2$ and $|\overline{D^-(A')}|=6$.
Since each $2$-partition in $\overline{D(A')}$ is a balanced $2$-partition, $\overline{D^-(A')}$ is also balanced.
Any such partition corresponds to a proper linear separation of $D(A')$.

Let us assume that $|\overline{D^+(A')}|= 3$ and $|\overline{D^-(A')}|=5$.
Since none of the $3$-partitions in $\overline{D(A')}$ is monochromatic, $\overline{D^+(A')}$ contains two points having the same color and one point having another color.
Any such partition corresponds to a linear separation of $D(A')$ such that five vectors lie on one side of the linear hyperplane and three vectors lie on the other side of the linear hyperplane.

Let us assume that $|\overline{D^+(A')}|= 1$ and $|\overline{D^-(A')}|=7$.
Any such partition corresponds to a linear separation of $D(A')$ such that five vectors lie in the one side of the linear hyperplane and three vectors lie in the other side of the linear hyperplane.

There also exists a unique partition of $\overline{D(A')}$ into $\overline{D^+(A')}$ and $\overline{D^-(A')}$ where $|\overline{D^+(A')}|=0$.
As shown in the proof of Observation~\ref{obs12}, such a partition corresponds to a proper linear separation of $D(A')$. 

The above argument shows that any linear separation of $D(A')$ contains at least three vectors in each of the open half-spaces created by the corresponding linear hyperplane.
Lemma~\ref{neigh} implies that $D(A')$ is a Gale transform of a $4$-dimensional  neighborly polytope. \qed
\end{proof}

\section{Improved Lower Bound on $\overline {cr}_d(K_{2d}^d)$}\label{sec11}
In this section, we improve the lower bound on $\overline {cr}_d(K_{2d}^d)$ to $\Omega\left(\dfrac{(4\sqrt{2}/3^{3/4})^d}{d}\right)$. In order to improve the lower bound on $\overline {cr}_d(K_{2d}^d)$, we need the following lemmas and the upper bound theorem~\cite{MM}.

\begin{lemma}~\cite[Proof of Theorem~$1$]{GS}
\label{lemgs}
Let $C'$ be a set containing $d+4$ points in general position in $\mathbb{R}^d$ for sufficiently large $d$. There exist at least $\floor{(d+4)/2}$ pairs of disjoint subsets $\{C'_{i1}, C'_{i2}\}$ of $C'$ for each $i$ satisfying $1 \leq i \leq \floor{(d+4)/2}$ such that the following properties hold. 
\begin{enumerate}
  \item $C'_{i1} \cup C'_{i2}=C'$ and $|C'_{i1}| , |C'_{i2}| \geq \floor{(d+2)/2}$.
  \item $Conv(C'_{i1})$ crosses the  $Conv(C'_{i2})$, i.e., $C'_{i1} \cap C'_{i2} = \emptyset$ and $Conv(C'_{i1}) \cap Conv(C'_{i2}) \neq \emptyset$. Note that $Conv(C'_{i1})$ is a $\left(|C'_{i1}|-1\right)$-simplex and $Conv(C'_{i2})$ is a $\left(|C'_{i2}|-1\right)$-simplex for sufficiently large $d$.

  \item There exist $C''_{i1} \subseteq C'_{i1}$ and $C''_{i2} \subseteq C'_{i2}$ such that $|C''_{i1}| , |C''_{i2}| \geq \floor{(d+2)/2}, |C''_{i1}|+ |C''_{i2}|=d+2 $ and the $\left(|C''_{i1}|-1\right)$-simplex  $Conv(C''_{i1})$ crosses the $\left(|C''_{i2}|-1\right)$-simplex  $Conv(C''_{i2})$. 
\end{enumerate}

\end{lemma}

\begin{lemma}~\cite{GS}
\label{lemgs1}
Consider a set $C$ that contains $2d$ points in general position in $\mathbb{R}^d$ for sufficiently large $d$. Let $C' \subset C$ be a subset such that $|C'|=d+4$.
Let $C'_1$ and $C'_2$ be two disjoint subsets of $C'$ such that $|C'_1|=c'_1, |C'_2|=c'_2, C'_1 \cup C'_2=C'$ and $c'_1, c'_2 \geq \floor{(d+2)/2}$.
If the $(c'_1-1)$-simplex formed by $C'_1$ and the $(c'_2-1)$-simplex formed by $C'_2$ form a crossing pair, then the $(d-1)$-simplex formed by a point set $B'_1\supset C'_1$ and the $(d-1)$-simplex formed by a point set $B'_2 \supset C'_2$ satisfying $B'_1 \cap B'_2 = \emptyset$, $|B'_1|, |B'_2| = d$ and $ B'_1 \cup B'_2 = C$ also form a crossing pair.
\end{lemma}

\begin{thm}~\cite{MM}
 Among all convex polytopes with a given dimension and number of vertices, the cyclic polytope has the largest possible number of faces of each dimension.    
\end{thm}

\noindent The following corollary can be deducted from the upper bound theorem, see~\cite{ZIE}.
\begin{corollary}\label{cor:upbound}\cite{ZIE}
The convex hull of an $n$-point set in $\mathbb{R}^d$ has at most $\dbinom{n-\ceil{d/2}}{n-d}+ \dbinom{n-\floor{d/2}-1}{n-d}$ facets.
\end{corollary}

\begin{remark}
    Let $L_1$ and $L_2$ be two affine subspaces of $\mathbb{R}^d$. Let us denote their associated vector spaces by $\overline{L_1}$ and $\overline{L_2}$, respectively. We say $L_1$ and $L_2$ are orthogonal complements of each other if $\overline{L_1}$  and $\overline{L_2}$ are orthogonal, and $\mathbb{R}^d= \overline{L_1} \oplus \overline{L_2}$.  
\end{remark}
  
\begin{lemma}
\label{lem:imp}
Let $\sigma$ be a $\floor{d/2}$-simplex, and $\tau$ be a $(d-1)$-simplex such that all the $d+\floor{d/2}+1$ points of $Vert(\sigma) \cup Vert(\tau)$ are in general position in $\mathbb{R}^d$ for sufficiently large $d$. At most  $O((3^{3/4}/\sqrt{2})^{d}/\sqrt{d})$ $\ceil{d/2}$-faces  of $\tau$ cross $\sigma$. 
\end{lemma}

\begin{proof}
Let us denote by $\tilde{\sigma}$ the affine hull of  $Vert(\sigma)$. Note that any $\ceil{d/2}$-face of $\tau$  that crosses $\sigma$ intersects $\tilde{\sigma}$. Also, note that no vertex of $\tau$ lies on $\tilde{\sigma}$ due to the general position of the points in $Vert(\sigma) \cup Vert(\tau)$. Let us consider the orthogonal complement space  ${\tilde{\sigma}}^{\perp}$ of $\tilde{\sigma}$. Let $\{\tau'_1, \tau'_2,\ldots, \tau'_l\}$ be the set of all $\ceil{d/2}$-faces of $\tau$ that intersects $\tilde{\sigma}$. Denote by $\mathcal{V}$  the set of vertices $\bigcup_{j=1}^l Vert(\tau'_j)$. Let us assume that at least one $\ceil{d/2}$-face of $\tau$ intersects $\tilde{\sigma}$. Then, we can assume that $\left|\mathcal{V}\right|= \ceil{d/2}+1+k$ for some positive integer $k$. Let $\mathcal{V}$ be the set $\mathcal{V}=\{v_1, v_2, \ldots, v_{\ceil{d/2}+1+k}\}$. We project the points in $Vert(\sigma) \cup \mathcal{V}$ onto ${\tilde{\sigma}}^{\perp}$. The set $Vert(\sigma)$ maps to a single point in ${\tilde{\sigma}}^{\perp}$. Without loss of generality, let us assume that  this point is the origin $O$. Denote by $\tilde{v}_i$  the projection of  $v_i$
for each $v_i \in \mathcal{V}$. Let $\tilde{\mathcal{V}}$ denote the set $\{\tilde{v}_2, \tilde{v}_2, \ldots, \tilde{v}_{\ceil{d/2}+1+k} \}$.

For a tuple $(i_1,i_2, \ldots, i_{\ceil{d/2}+1})$, where $1 \leq i_1 < i_2< \ldots < i_{\ceil{d/2}+1} \leq \ceil{d/2}+1+k$, the convex hull of $\{\tilde{v}_{i_1},\tilde{v}_{i_2}, \ldots, \tilde{v}_{i_{\ceil{d/2}+1}}\}$ 
contains the origin $O$ if and only if $\ceil{d/2}$-simplex spanned by $\{v_{i_1}, v_{i_2}, \ldots, v_{i_{\ceil{d/2}+1}}\}$ intersects $\tilde{\sigma}$. Let us denote by $\hat{v_i}$ the vector $O\tilde{v}_i$. Consider the vector configuration $\hat{\mathcal{V}}=\{\hat{v}_1, \hat{v}_2, \ldots,$ $\hat{v}_{\ceil{d/2}+1+k}\}$. The following observations hold.

 \begin{observation}
 \label{obs:generalpos}
 Any subset of $\ceil{d/2}$ vectors of $\hat{\mathcal{V}}$ spans $\mathbb{R}^{\ceil{d/2}}$. 
\end{observation}
  For the sake of contradiction,  let us assume that the set of vectors $\{\hat{v}_{i_1}, \hat{v}_{i_2}, \ldots, \hat{v}_{i_{\ceil{d/2}}}\}$ spans some $\mathbb{R}^q$ where $q < \ceil{d/2}$. Then, $d+1$ points $\{v_{i_1},v_{i_2}, \ldots, v_{i_{\ceil{d/2}}} \}\cup Vert(\sigma)$ lie in a subspace whose dimension is  $\floor{d/2}+q \leq d-1$. This contradicts the fact that $Vert(\sigma) \cup Vert(\tau)$ are in general position in $\mathbb{R}^d$.

  Note that Observation~\ref{obs:generalpos} implies that each point of $\tilde{\mathcal{V}}$ is  distinct.  It also implies that  if the convex hull of $\{\tilde{v}_{i_1},\tilde{v}_{i_2}, \ldots, \tilde{v}_{i_{\ceil{d/2}+1}}\}$ contains the origin $O$, then the origin $O$ lies in the relative interior of  the convex hull of $\{\tilde{v}_{i_1},\tilde{v}_{i_2}, \ldots, \tilde{v}_{i_{\ceil{d/2}+1}}\}$.
  
   \begin{observation} 
   $\hat{\mathcal{V}}$ is a totally cyclic vector configuration.  
\end{observation}
  For $1 \leq i \leq \ceil{d/2}+k+1$, each $\tilde{v}_i$ is a vertex of a convex hull of some $\ceil{d/2}+1$ points of $\tilde{\mathcal{V}}$ which contains the origin $O$.
  
 These observations along with Lemma~\ref{lem:conditiongale} imply that with proper scaling $\hat{\mathcal{V}}$ is a Gale transformation of some sequence of points $Q=\langle q_1, q_2, \ldots, q_{\ceil{d/2}+1+k} \rangle$ having $\ceil{d/2}+1+k$ points in $\mathbb{R}^k$. Lemma~\ref{lem:generalpos} and Observation~\ref{obs:generalpos} together imply that these $\ceil{d/2}+1+k$ points of $Q$ are in general position in $\mathbb{R}^k$. Since the points in $Q$  are in general position in $\mathbb{R}^k$, Lemma~\ref{lem:facelemma} implies that  the convex hull of $\{\tilde{v}_{i_1},\tilde{v}_{i_2}, \ldots, \tilde{v}_{i_{\ceil{d/2}+1}}\}$ contains the origin $O$ if and only if the $k$ points in $Q \setminus \{q_{i_1}, q_{i_2}, \ldots, q_{i_{\ceil{d/2}+1}}\} $ span a $(k-1)$-dimensional face of the convex hull of $Q$. By Corollary~\ref{cor:upbound}, the maximum number of $(k-1)$-dimensional faces of the convex hull of $Q$ is $\dbinom{\ceil{d/2}+1+k-\ceil{k/2}}{\ceil{d/2}+1+k-k}+ \dbinom{\ceil{d/2}+1+k-\floor{k/2}-1}{\ceil{d/2}+1+k-k}=$ $\dbinom{\ceil{d/2}+1+\floor{k/2}}{\ceil{d/2}+1}+ \dbinom{\ceil{d/2}+\ceil{k/2}}{\ceil{d/2}+1}$. This quantity increases as $k$ increases. The maximum value $k$ can take is $\floor{d/2}-1$. By setting $k=\floor{d/2}-1$ and using Stirling's approximation, we obtain the upper bound   $O((3^{3/4}/\sqrt{2})^{d}/\sqrt{d})$ on the number of $\ceil{d/2}$-faces  of $\tau$ that intersect $\sigma$.  \qed
\end{proof}

\begin{remark}
Given $n$ points in general position in $\mathbb{R}^m$,  the number of $m$-simplices that are spanned by these $n$ points and contain the origin in their interiors is at most $\dbinom{\floor{\dfrac{n+m}{2}}}{m+1}+\dbinom{\ceil{\dfrac{n+m}{2}}}{m+1}$~\cite[Remark~$4.2$]{WW}. B{\'a}r{\'a}ny ~\cite{BAR} proved  this statement. \par
In the proof of Lemma~\ref{lem:imp}, the points in $\tilde{\mathcal{V}}$ are not necessarily in general position in $\mathbb{R}^{\ceil{d/2}}$. Due to this fact, we could not use the theorem of  B{\'a}r{\'a}ny~\cite{BAR} in our proof. Though the proof of Lemma~\ref{lem:imp} is similar to the proof of the above mentioned statement, for the sake of completeness we presented the proof in detail.
\end{remark}

\noindent \textbf{Proof of Theorem~\ref{thm0}:} Let $V=\{v_1,v_2, \ldots, v_{2d}\}$ be the vertices of $K_{2d}^d$ in its $d$-dimensional rectilinear drawing. Note that the points in $V$ are in general position in $\mathbb{R}^d$.
Let $E$ be the set of $(d-1)$-simplices that correspond to hyperedges of the drawing.
Choose  a subset of $V' \subset V$ having $d+4$ points.
Lemma~\ref{lemgs} implies that there exist $\floor{(d+4)/2}$ pairs of subsets $\{V'_{i1}, V'_{i2}\}$ for each $i$ satisfying $1 \leq i \leq \floor{(d+4)/2}$ such that all the three conditions mentioned in Lemma~\ref{lemgs}  hold.

It follows from Lemma~\ref{lemgs1} that each such crossing pair of a $(|V'_{i1}|-1)$-simplex and a $(|V'_{i2}|-1)$-simplex can be extended to  at least $\dbinom{d-4}{d-\floor{(d+2)/2}}= \Omega\left({2^d}/{\sqrt{d}}\right)$ crossing pairs of $(d-1)$-simplices corresponding to  the crossing pairs of hyperedges in $E$.
Therefore, the total number of crossing pairs of hyperedges, originated from a particular choice of $V'$, in a $d$-dimensional rectilinear drawing of $K_{2d}^d$ is at least $\floor{{(d+4)}/{2}}\Omega\left({2^d}/{\sqrt{d}}\right)=\Omega \left(2^d \sqrt{d}\right)$. \par

We can choose $V'$ in $\dbinom{2d}{d+4}= \Theta\left(4^d/\sqrt{d}\right)$ ways.
On the one hand, there exist $\Omega\left(2^d\sqrt{d}\right)$ crossing pairs of hyperedges in a $d$-dimensional rectilinear drawing of $K_{2d}^d$ for each choice of $V'$. On the other hand, each crossing pair of hyperedges may originate from the different choices of subsets having $d+4$ points from $V$. Next, we estimate an upper bound on the number of such choices. 

Let  $e_k,e_l \in E$ be a crossing pair of hyperedges. Let $V(e_k)$, $V(e_l)$ respectively denote the set of vertices of $e_k$ and $e_l$. Note that $|V(e_k)|=|V(e_l)|=d$,  $V(e_k) \cup V(e_l)= V$ and $V(e_k) \cap V(e_l)= \emptyset$. 

\noindent Lemma~\ref{lem:imp} implies that there can be at most $2 \dbinom{d}{\floor{d+2/2}}O((3^{3/4}/\sqrt{2})^{d}/\sqrt{d})= O((3^{3/4} \sqrt{2})^d/d)$  pairs of $\{V''_{i1},V''_{i2}\}$ such that following conditions hold.
\begin{enumerate}
\item $V''_{i1} \subset V(e_k)$, $V''_{i2}  \subset V(e_l)$ or $V''_{i1} \subset V(e_l)$, $V''_{i2} \subset V(e_k)$.  

\item $|V''_{i1}| , |V''_{i2}| \geq \floor{(d+2)/2}, |V''_{i1}|+ |V''_{i2}|=d+2$.

\item $Conv(V''_{i1})$ crosses $Conv(V''_{i2})$.
\end{enumerate}

Let us assume that the crossing pair of hyperedges $e_k,e_l \in E$ originates from a $d+4$ sized subset $V'_i$ of $V$.  Note that $V'_i$ contains two subsets $V'_{i1}$ and $V'_{i2}$ such that following conditions hold.

\begin{enumerate}
    \item $|V'_{i1}|,|V'_{i2}| \geq \floor{d+2/2}$, $V'_{i1} \cup V'_{i2}=V'$ and $V'_{i1} \cap V'_{i2}=\emptyset$.
    
    \item There exist either $V''_{i1} \subseteq V'_{i1}\subset V(e_k)$, $V''_{i2} \subseteq V'_{i2} \subset V(e_l)$ or $V''_{i1} \subseteq V'_{i1}\subset V(e_l)$, $V''_{i2} \subseteq V'_{i2} \subset V(e_k)$ 
    such that $|V''_{i1}| , |V''_{i2}| \geq \floor{(d+2)/2}, |V''_{i1}|+ |V''_{i2}|=d+2$  and $Conv(V''_{i1})$ crosses $Conv(V''_{i2})$. 

\end{enumerate}

Since each $V'_i$ contains a  pair  $\{V''_{i1},V''_{i2}\}$ and  each pair $\{V''_{i1},V''_{i2}\}$ can be extended in $O(d^2)$ ways to  a set of size $d+4$, a crossing pair of hyperedges can originate from $O((3^{3/4} \sqrt{2})^d/d)\times O(d^2)= O((3^{3/4} \sqrt{2})^dd)$ distinct subsets $V'_i \subset V$ of size $d+4$. This implies that there exist at least $\dfrac{\Omega \left(2^d \sqrt{d}\right)\Theta\left(4^d/\sqrt{d}\right)}{O((3^{3/4} \sqrt{2})^dd)}=\Omega\left(\dfrac{(4\sqrt{2}/3^{3/4})^d}{d}\right)$ crossing pairs of hyperedges in any $d$-dimensional rectilinear drawing of $K_{2d}^d$. \qed

\section {Maximum Rectilinear Crossing Number of Complete
$4$-uniform Hypergraphs}\label{sec3}

In this section, we prove that the maximum $4$-dimensional rectilinear crossing number of $K_{n}^4$ is $13 \dbinom{n}{8}$. This result proves Anshu et al.'s conjecture~\cite{AGS} affirmatively for $d=4$. 
We also produce a family of $4$-dimensional rectilinear drawings of $K_{n}^4$ having $13\dbinom{n}{8}$ crossing pairs of hyperedges. In order to prove these results, we first define the concept of the order type.

 \begin{figure*}[h]
  \centering
   \subfloat[]{\label{fig:rahul_fig1}\resizebox{0.35\textwidth}{!}{\begin{tikzpicture}[
	> = stealth, 
	auto,
	line width=0.25mm,
	node distance = 0.15cm, 
	circleshaded/.style={circle, draw=black, fill={rgb,255:red,220; green,220; blue,220}, very thick, scale=1, minimum size=5mm},
	circlewhite/.style={circle, draw=black, fill=none, very thick, scale=1, minimum size=1mm},
	]
      

	\node[circlewhite, label=left:{\Large $s'_i$}] (1) at (0,0) {};
	\node[circlewhite, label=above:{\Large $s'_j$}] (2) at (2,1) {};
	\node[circlewhite, label=right:{\Large $s'_k$}] (3) at (4,0) {};

	\path [] (1) edge (2);
	\path [] (2) edge (3);

	\draw[->,line width=1.5pt] (0.4, -0.2) .. controls(2,0.8) .. (3.6, -0.2);

      
\end{tikzpicture}}}
   \hspace{0.5cm}
   \subfloat[]{\label{fig:rahul_fig2}\resizebox{0.25\textwidth}{!}{\begin{tikzpicture}[
	> = stealth, 
	auto,
	line width=0.25mm,
	node distance = 0.15cm, 
	circleshaded/.style={circle, draw=black, fill={rgb,255:red,220; green,220; blue,220}, very thick, scale=1, minimum size=5mm},
	circlewhite/.style={circle, draw=black, fill=none, very thick, scale=1, minimum size=1mm},
	]
      

	\node[circlewhite, label=left:{\Large $s'_i$}] (1) at (0,0) {};
	\node[circlewhite, label=right:{\Large $s'_j$}] (2) at (2,1) {};
	\node[circlewhite, label=above:{\Large $s'_k$}] (3) at (0,2) {};

	\path [] (1) edge (2);
	\path [] (2) edge (3);

	\draw[->,line width=1.5pt] (0.2, 0.4) .. controls(1.8,1) .. (0.2, 1.6);

      
\end{tikzpicture}}}
  \caption{Possible orientations of a triplet in $\mathbb{R}^2$ }
  \label{fig:rahul_figure}
 \end{figure*}
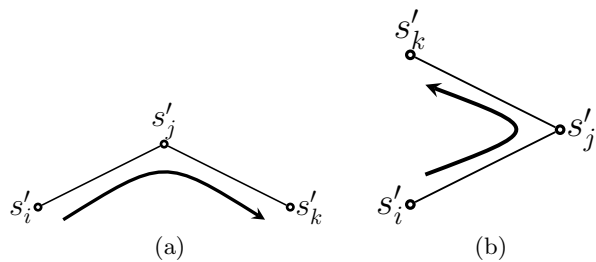

\begin{definition}[Order type]
Consider a sequence of points $S'= \langle s'_1,s'_2,$ $\ldots, s'_n \rangle$ where points are in general position in $\mathbb{R}^2$. We say an ordered triple $\langle s'_i, s'_j, s'_k \rangle$ with $i<j<k$ is oriented clockwise (counter-clockwise) if the point $s'_k$ lies in the clockwise (counter-clockwise) direction with respect to the line segment ${\rm  s'_is'_j}$ directed from $s'_i$ to  $s'_j$, see Figure~\ref{fig:rahul_figure}.
The order type of $S'$ is a mapping which assigns an orientation (clockwise or counter-clockwise) to each ordered triple $\langle s'_i, s'_j, s'_k \rangle$.
\end{definition}

\noindent Consider two sequences of points $S'= \langle s'_1,s'_2,$ $\ldots, s'_n \rangle$ and $S''=\langle s''_1,s''_2,$ $\ldots, s''_n \rangle$ in $\mathbb{R}^2$, such that the points in both the sequences are in general position. 
Sequences $S'$ and $S''$ are said to have same order type if for any indices $i<j<k$ the orientation of $\langle s'_i, s'_j, s'_k \rangle$ is same as the orientation of $\langle s''_i, s''_j, s''_k \rangle$.\\ 

\noindent Suppose that two sequences of points $S'=\langle s'_1,s'_2,$ $\ldots, s'_{2n} \rangle$ and $S''=\langle s''_1,s''_2, \ldots, s''_{2n} \rangle$ in $\mathbb{R}^2$ have same order type.
Consider a coloring $C$ where $n$ points of $S'$ are colored red, and the remaining $n$ points are colored violet.
The indices of red points are also the same in $S'$ and $S''$, implying that the indices of violet points are also the same.
For each tuple $(i_1,i_2, \ldots, i_{2m})$, where $1 \leq i_1 < i_2< \ldots < i_{2m} \leq 2n$, $\{s'_{i_1},s'_{i_2}, \ldots, s'_{i_{2m}}\}$ is a balanced $2m$-partition of $S'$ if and only if $\{s''_{i_1},s''_{i_2}, \ldots, s''_{i_{2m}}\}$ is a balanced $2m$-partition of $S''$~\cite{ZIE}.
There are infinitely many point configurations having $n$ points in general position in $\mathbb{R}^2$. 
There are only finitely many order types for such point configurations. 
We can think of order types as equivalence classes.
The point configurations that have the same order type share many combinatorial and geometric properties.
Aichholzer et al.~\cite{ALCHO1,ALCHO} created a database which contains all order types of eight points in general position in $\mathbb{R}^2$.
We use those point sets in  proofs of Theorem~\ref{thm1} and Lemma~\ref{lem:114}.\\

\noindent\textbf{Proof of Theorem~\ref{thm1}.} Let us consider all order types of the eight points in general position in $\mathbb{R}^2$. Aichholzer et al.~\cite{ALCHO1,ALCHO} listed all possible 3315 order types with their representative elements. 
Let us denote the point sequence corresponding to the $i^{th}$ order type with $o_i$. 
We also generate all possible colorings of a sequence of eight points where four of the points are red, and the remaining points are violet.
There are $\dbinom{8}{4}=70$ such colorings. 
Each coloring can be represented as an $8$-bit binary string with an equal number of zeroes and ones. 
Let us represent the $j^{th}$ coloring in lexicographical order by $c_j$.
We consider the point sequence of each order type and color it according to all the seventy possible ways such that there is an equal number of red and violet points in each coloring.
Formally, we consider the set $O_C=\{(o_i,c_j): 1 \leq i \leq 3315, 1 \leq j \leq 70 \}$ containing all possible pairs of $(o_i,c_j)$ for each $i$ satisfying $1 \leq i \leq 3315$, and $1 \leq j \leq 70$.\par

Consider a $4$-dimensional rectilinear drawing of $K_8^4$ where the vertices of $K_8^4$ are points in general position in $\mathbb{R}^4$. 
Let us denote these vertices by $V=\{v_1,v_2, \ldots, v_8\}$.
Consider a Gale transform $D(V)$ of $V$. 
Lemma~\ref{bjection} implies that the number of proper linear separations of $D(V)$ is equal to the number of crossing pairs of hyperedges in this particular drawing of $K_8^4$ since there exists a bijection between crossing pairs of hyperedges and proper linear separations of $D(V)$. 
Consider an affine Gale diagram $\overline{D(V)}$ having four red and four violet points such that all the eight points are in general position in $\mathbb{R}^2$. 
Observation~\ref{obs11} ensures such a $\overline{D(V)}$ always exists. 
Observation~\ref{obs12} ensures that the number of proper linear separations of $D(V)$ is equal to the total number of balanced $2$-partitions of $\overline{D(V)}$ plus the total number of balanced $4$-partitions of $\overline{D(V)}$ plus one. 
Note that $\overline{D(V)}$ is equivalent to one of the elements of $O_C$. 
Note that all elements of $O_C$ need not be an affine Gale diagram of some eight points in $\mathbb{R}^4$. 
Consider the point sequence $o_i$ under the coloring $c_j$. If there exists a monochromatic $4$-partition of $o_i$ under the colouring $c_j$, then $(o_i,c_j)$ is a projection of an acyclic vector configuration, and it cannot be an affine Gale diagram of any set of eight points in $\mathbb{R}^4$. 
We find the maximum of the sum of the total number of balanced $2$-partitions and the total number of balanced $4$-partitions over all members of $O_C$ by analyzing each of its members.
We wrote a program for this purpose in Python 3.7.1.
The code has been made open source as well.\footnote{https://github.com/ayan-iiitd/maximum-rectilinear-crossing-number-of-uniform-hypergraphs.git}
We find the maximum to be twelve. 
Observation~\ref{obs12} implies that the maximum number of proper linear separations of $D(V)$ is $12+1=13$. 
Lemma~\ref{bjection} implies that the maximum number of crossing pairs of hyperedges in any $4$-dimensional rectilinear drawing of $K_8^4$ is thirteen.\par

Consider a $4$-dimensional rectilinear drawing of $K_n^4$ where all the vertices are placed on the $4$-dimensional moment curve.

Anshu et al.~\cite{AGS} showed that in this drawing, every $K_8^4$ has 13 crossing pairs of hyperedges. 
Since the crossing pairs of hyperedges spanned by a set of eight vertices are distinct from the crossing pairs of hyperedges spanned by another set of eight vertices, the above argument shows that
\emph{max-}$\overline {cr}_4(K_{n}^4) = 13 \dbinom{n}{8}$. \qed

\begin{lemma} \label{lem:114}
Consider a $4$-dimensional neighborly polytope $P$ having $n$ vertices such that all the vertices of $P$ are in general position in $\mathbb{R}^4$.
Consider a $4$-dimensional rectilinear drawing of $K^4_{n}$ such that the vertices of $K^4_{n}$ are placed at the vertices of $P$.
The number of crossing pairs of hyperedges in this $4$-dimensional rectilinear drawing of  $K^4_{n}$ is $13\dbinom{n}{8}$.
\end{lemma}

\begin{proof}
As mentioned in the proof of Theorem~\ref{thm1}, let us consider the set $O_C=\{(o_i,c_j): 1 \leq i \leq 3315, ~ 1 \leq j \leq 70 \}$ containing all possible pairs of $(o_i,c_j)$ for each $i$ satisfying $1 \leq i \leq 3315$ and each $j$ satisfying $1 \leq j \leq 70$. \par

Consider a $4$-dimensional rectilinear drawing of $K_8^4$ where the vertices of $K_8^4$ are placed at the vertices of a $4$-dimensional neighborly polytope whose vertices are in general position in $\mathbb{R}^4$. 
Let us denote these vertices by $V=\{v_1,v_2, \ldots, v_8\}$.
Consider a Gale transform $D(V)$ of $V$. 
Consider an affine Gale diagram $\overline{D(V)}$ having four red and four violet points such that all the eight points are in general position in $\mathbb{R}^2$.
Lemma~\ref{obs:13} gives us the necessary and sufficient conditions for $(o_i,c_j)$ to be a Gale transform of a $4$-dimensional neighborly polytope whose vertices are in general position in $\mathbb{R}^4$.

 Let us consider all pairs $(o_i,c_j)$ such that they satisfy the three conditions mentioned in Lemma~\ref{obs:13}. 
Let us denote this collection by $O'$.\\


$O'=\{(o_i,c_j): (o_i,c_j)$\emph{~follows the three conditions mentioned in Lemma~\ref{obs:13}$\,\}$}.\\

Note that $\overline{D(V)}$ is equivalent to one of the elements of $O'$. Also, note that each  pair composed by an order type and a coloring corresponding to the members of $O'$  is equivalent to an affine Gale diagram of a $4$-dimensional neighborly polytope having all its eight vertices in general position in $\mathbb{R}^4$. \par

We calculate the sum of the total number of balanced $2$-partitions and the total number of balanced $4$-partitions over all members of $O'$ by analyzing each of its members.
We wrote a program for this purpose in Python 3.7.1, and it is also available in the aforementioned GitHub repository.
We find the value to be twelve for all members of $O'$. \par
Observation~\ref{obs12} implies that the number of proper linear separations of $D(V)$ is $12+1=13$. 
This implies that there exist thirteen crossing pairs of hyperedges in a $4$-dimensional rectilinear drawing of $K_8^4$ when the vertices of $K_8^4$ are placed at the vertices of a $4$-dimensional neighborly polytope having all its eight in general position in $\mathbb{R}^4$.\par

Consider a $4$-dimensional neighborly polytope $P$ having $n$ vertices such that all the vertices of $P$ are in general position in $\mathbb{R}^4$.
Consider a $4$-dimensional rectilinear drawing of  $K^4_{n}$ such that the vertices of $K^4_{n}$ are placed at the vertices of $P$.
Consider any subset $P'$ of the vertex set of $P$ having size eight.
Note that the $4$-dimensional polytope spanned by the vertices of $P'$ is also a neighborly polytope.
This implies that in such a drawing every copy of $K_8^4$ has 13 crossing pairs of hyperedges. 
Since the crossing pairs of hyperedges spanned by a set of eight vertices are distinct from the crossing pairs of hyperedges spanned by another set of eight vertices, the above argument shows that the number of crossing pairs of hyperedges in a $4$-dimensional rectilinear drawing of $K_n^4$ is $13 \dbinom{n}{8}$ if the vertices of $K^4_{n}$ are placed at the vertices of a $4$-dimensional neighborly polytope having all its vertices in general position in $\mathbb{R}^4$. \qed
\end{proof}

\section{Maximum Rectilinear Crossing Number of Complete $d$-partite $d$-uniform Hypergraphs}\label{sec4}

In this section, we prove that \emph{max-}$\overline {cr}_d(K^d_{d \times n})= (2^{d-1}-1) {\dbinom{n}{2}}^d$. In order to prove this result, we first prove that the maximum $d$-dimensional rectilinear crossing number of $K^d_{d \times 2}$ is $2^{d-1}-1$ in Lemma~\ref{lem:1}. 
We then create a $d$-dimensional rectilinear drawing of $K^d_{d \times n}$ such that each of the ${\dbinom{n}{2}}^d$ induced sub-hypergraphs, which is isomorphic to $K^d_{d \times 2}$, spans $2^{d-1}-1$ crossing pairs of hyperedges. 
Since the crossing pairs of hyperedges spanned by a copy of $K^d_{d \times 2}$ are distinct from the crossing pairs of hyperedges spanned by another copy of $K^d_{d \times 2}$, this implies that {max-}$\overline {cr}_d(K^d_{d \times n})= (2^{d-1}-1) {\dbinom{n}{2}}^d$. \par

Next, we state three lemmas which are used in the proof of Lemma~\ref{lem:1}.
\begin{lemma}
~\cite{BR} \label{lem:dpach}
Let $p_1 \prec p_2 \prec 
\ldots \prec p_{\floor{\dfrac{d}{2}}+1}$ and 
$q_1 \prec q_2 \prec \ldots \prec 
q_{\ceil{\dfrac{d}{2}}+1}$ be two
distinct point sequences on the 
$d$-dimensional moment curve such that 
$p_i \neq q_j$ for any $1 \leq i \leq \floor{\dfrac{d}{2}}+1$
and $1 \leq j \leq \ceil{\dfrac{d}{2}}+1$.
The $\floor{\dfrac{d}{2}}$-simplex and the $\ceil{\dfrac{d}{2}}$-simplex,
formed respectively by these point sequences,
\emph{cross} if and only if every interval $(q_j , q_{j+1})$
contains exactly one $p_i$ and every interval $(p_i , p_{i+1})$
contains exactly one $q_j$. 
\end{lemma}

\begin{lemma}
~\cite{DP} \label{lem:dpach1}
Let $P$ and $Q$ be two vertex-disjoint $(d-1)$-simplices such that each of the 
$2d$ vertices belonging to these simplices lies on the $d$-dimensional moment curve. 
If $P$ and $Q$ cross, then there exists a $\floor{\dfrac{d}{2}}$-simplex $U 
\subsetneq P$ and another $\ceil{\dfrac{d}{2}}$-simplex $V \subsetneq Q$ such that $U$ and $V$ cross. 
\end{lemma}

\begin{lemma}
\label{AA}~\cite{AA}
Let us consider $d$ pairwise disjoint sets, each having $n$ points in $\mathbb{R}^d$, such that all $dn$ points are in general position.
Then there exist $n$ pairwise disjoint $(d-1)$-simplices such that each simplex has one vertex from each set.
\end{lemma}

\begin{lemma} \label{lem:1}
For $d \geq 2$, the maximum $d$-dimensional rectilinear crossing number of $K^d_{d \times 2}$ is $2^{d-1}-1$.
\end{lemma}

\begin{proof}
Consider the hypergraph $K^d_{d \times 2}$. For each $i$ satisfying $1 \leq i \leq d$, let us denote the $i^{th}$ part of the vertex set of $K^d_{d \times 2}$ by $C_i$.
Let $\{p_{c_{i}}, p'_{c_{i}} \}$ denote the set of two vertices in $C_i$.
Let $A$ be a set of $d$ vertices of $K^d_{d \times 2}$ such that each vertex of $A$ is from different parts of $K^d_{d \times 2}$.
Let $B$ be the set of the remaining vertices of $K^d_{d \times 2}$.
Note that $|B|=d$ and each vertex of $B$ is from different parts of $K^d_{d \times 2}$. 
The number of unordered pairs $\{A,B\}$ is $\dfrac{1}{2}2^d= 2^{d-1}$.
Our goal is to find the maximum number of unordered pairs, $\{A,B\}$ such that the $(d-1)$-simplex formed by the vertices of $A$ forms a crossing with the $(d-1)$-simplex formed by the vertices of $B$.
Lemma~\ref{AA} implies that in any $d$-dimensional rectilinear drawing of $K^d_{d \times 2}$, there exists a pair of disjoint simplices such that each simplex has one vertex from each part of $K^d_{d \times 2}$.
This implies the maximum number of unordered pairs $\{A,B\}$ such that $(d-1)$-simplex formed by the vertices of $A$ forms a crossing with the $(d-1)$-simplex formed by the vertices of $B$ is $2^{d-1}-1$.\par

\begin{figure}[!ht]
  \centering
  \includegraphics{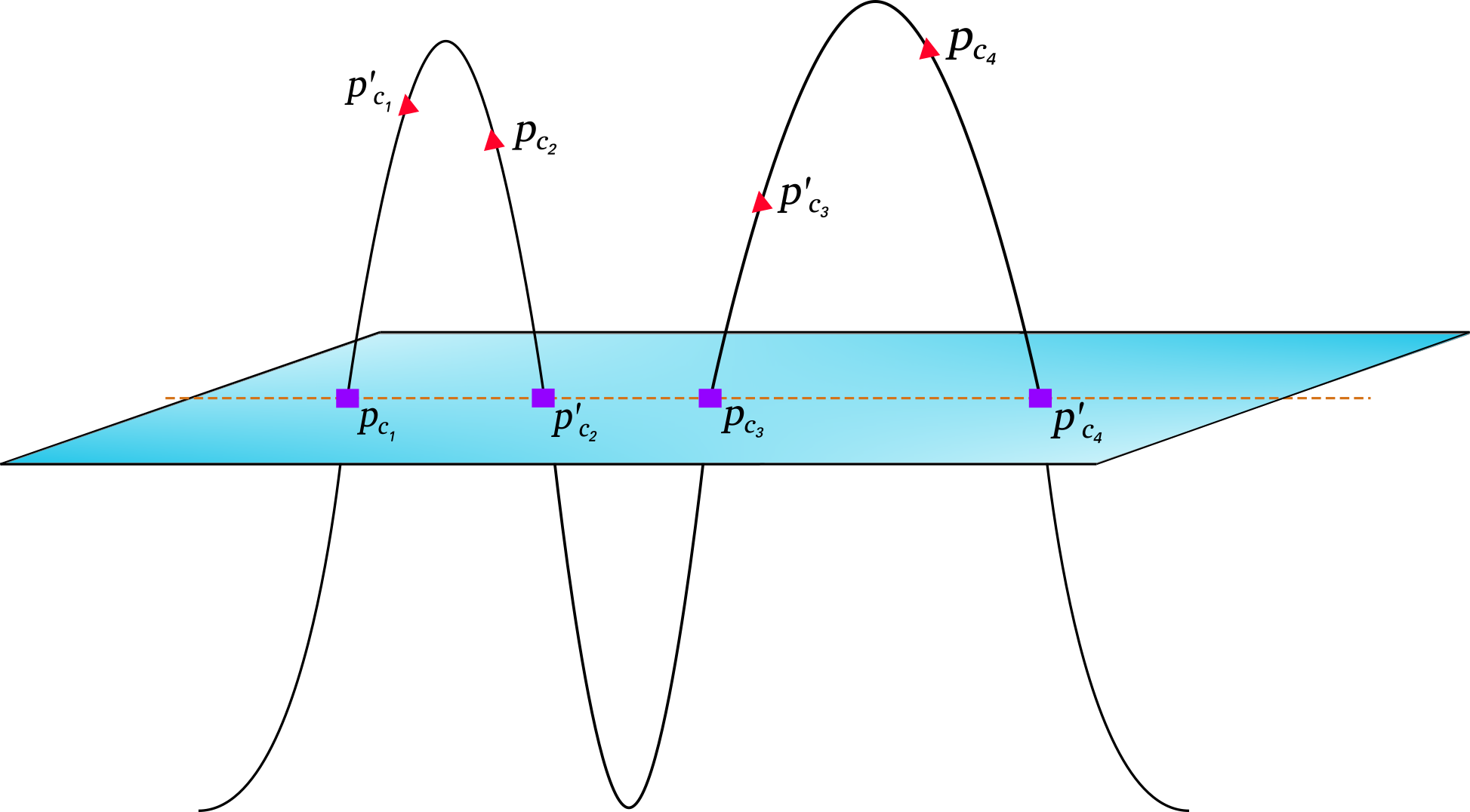}
  \caption{Non-crossing pair of hyperedges of $K^4_{4 \times 2}$}
  \label{fig:1}
\end{figure}

Let us consider a particular $d$-dimensional rectilinear drawing of $K^d_{d \times 2}$ that achieves the above mentioned bound. 
In this particular drawing, the vertices of $K^d_{d \times 2}$ are placed on the $d$-dimensional moment curve such that they satisfy the ordering $p_{c_{1}} \prec p'_{c_{1}} \prec p_{c_{2}} \prec p'_{c_{2}}\ldots \prec p_{c_{d-1}} \prec p'_{c_{d-1}} \prec p_{c_{d}} \prec p'_{c_{d}}$.

Let us assume that for any  pair $\{A,B\}$, $A$ contains the first vertex, i.e., $p_{c_{1}}$. 
Given a pair $\{A, B\}$, the vertices of $A$ create $d$ segments of the $d$-dimensional moment curve.
We call each segment a bucket. The $i^{th}$ bucket $b_i$ contains the set of points on the moment curve between $\{p_{c_i}, p'_{c_i}\} \cap A$ and $\{p_{c_{i+1}}, p'_{c_{i+1}}\} \cap A$ where $1 \leq i \leq d-1$. Note that the last bucket $b_d$ has only one endpoint created by the last vertex (according to the order mentioned above) of $A$ and contains all the points over the $d$-dimensional moment curve which succeed the last vertex of $A$. Note that the points on the $d$-dimensional moment curve which precede $p_{c_{1}}$ are not part of any bucket.
 
Since, both $A$ and $B$ contain exactly one vertex from each part of the vertex set, the following properties hold.

\begin{itemize}
\item The first bucket contains either one vertex or two vertices of $B$, but it can never be empty.
\item For each $i$ satisfying $2 \leq i \leq d-1$, each bucket $b_i$ can contain no vertex of $B$, one vertex of $B$ or two vertices of $B$ depending upon the endpoints of the bucket.
The last bucket contains either no vertex or one vertex of $B$.
\item For any pair of consecutive buckets, both of them cannot contain two vertices of $B$.
\end{itemize}

Lemmas~\ref{lem:dpach} and~\ref{lem:dpach1} together imply that $Conv(A)$ and $Conv(B)$ do not \emph{cross} if and only if there does not exist an alternating chain of $d+2$ vertices as mentioned in Lemma~\ref{lem:dpach}. \par

Note that to avoid such an alternating chain of $d+2$ vertices exactly $\floor{d/2}$ buckets should be empty since every bucket can contain at most two vertices of $B$ and all the $d$ vertices of $B$ should be partitioned into the $d$ buckets.
Also, note that any two non-empty buckets are not consecutive, and the first bucket is not empty.\\

When $d$ is even, this implies that each of the odd-numbered buckets contains two vertices, and even-numbered buckets are empty. 
The only pair $\{A, B\}$ such that the $(d-1)$-simplex formed by the vertices of $A$ does not form a crossing with the $(d-1)$-simplex formed by the vertices of $B$ is the following:\\



\noindent \scalebox{1}{$A=\{p_{c_{1}}, p'_{c_{2}}, p_{c_{3}}, p'_{c_{4}}, \ldots, p_{c_{d-1}}, p'_{c_{d}} \}, B=\{p'_{c_{1}}, p_{c_{2}}, p'_{c_{3}}, p_{c_{4}}, \ldots, p'_{c_{d-1}}, p_{c_{d}} \}$}.\\

When $d$ is odd, the last bucket should contain exactly one vertex of $B$.
Otherwise, we can form an alternating chain of $d+2$ vertices since at least $\floor{d/2}+1$ of the first $d-1$ buckets are non-empty.
This implies that for odd $d$, all the even-numbered buckets are empty and each of the odd-numbered buckets contains two vertices except the last bucket which contains one vertex.
The only pair $\{A, B\}$ such that the $(d-1)$-simplex formed by the vertices of $A$ does not form a crossing with the $(d-1)$-simplex formed by the vertices of $B$ is the following:\\

\noindent \scalebox{0.9}{$A=\{p_{c_{1}}, p'_{c_{2}}, p_{c_{3}}, p'_{c_{4}}, \ldots, p_{c_{d-2}}, p'_{c_{d-1}}, p_{c_{d}} \}, B=\{p'_{c_{1}}, p_{c_{2}}, p'_{c_{3}}, p_{c_{4}}, \ldots, p'_{c_{d-2}}, p_{c_{d-1}}, p'_{c_{d}} \}$}. \qed
\end{proof}

\noindent\textbf{Proof of Theorem~\ref{thm2}.}
For each $i$ satisfying $1 \leq i \leq d$, let $C_i$ denote the $i^{th}$ partition of the vertex set of $K^d_{d \times n}$.
Let $\{p^i_1, p^i_2, \ldots, p^i_n \}$ denote the set of $n$ vertices in $C_i$.
Consider the following arrangement of the vertices of $K^d_{d \times n}$ on the $d$-dimensional moment curve:

\begin{itemize}
\item Any vertex of $C_i$ precedes any vertex of $C_j$ if $i < j$.
\item For each $i$ satisfying $1 \leq i \leq d$, $p^i_l \prec p^i_m$ if $l < m$.
\end{itemize}
Consider any induced sub-hypergraph of $K^d_{d \times n}$ which is isomorphic to $K^d_{d \times 2}$.
In this particular $d$-dimensional rectilinear drawing of $K^d_{d \times n}$, the vertices of the sub-hypergraph follow the same ordering mentioned in the proof of Lemma~\ref{lem:1}, implying that each of them contains $2^{d-1}-1$ crossing pairs of hyperedges and the maximum $d$-dimensional rectilinear crossing number of $K^d_{d \times n}$ is $(2^{d-1}-1){\dbinom{n}{2}}^d$. \qed

\section{On the Maximum Rectilinear Crossing Number of General Hypergraphs}\label{sec5}

In this section, we turn our focus on finding the maximum $d$-dimensional rectilinear crossing number of an arbitrary $d$-uniform hypergraph $H$. 
Given $H$ and an integer $l$, we show that finding the existence of a $d$-dimensional rectilinear drawing $D$ of $H$ having at least $l$ crossing pairs of hyperedges is NP-hard. \par

We reduce the Max E$_k$-set splitting problem, which is known to be NP-hard, to our problem. 
Given a $k$-uniform hypergraph $H'=(V',E')$ and an integer $c$, the decision version of Max E$_k$-set splitting asks whether there exists a partition of $V'$ in two parts such that at least $c$ hyperedges of $E'$ contain at least one vertex from both the parts.
Lov{\'a}sz~\cite{LOV} proved that given a $k$-uniform hypergraph $H'=(V',E')$, deciding whether $H'$ is $2$-colorable is NP-hard when $k \geq 3$. 
For $k \geq 3$, this problem is a special case of the decision version of the Max  E$_k$-set splitting where $c=|E'|$. 
This implies that for $k \geq 3$, the decision version of the Max  E$_k$-set splitting problem is also NP-hard.  
Note that the Max  E$_k$-set splitting problem is the same as the Max-Cut problem for $k=2$.
It has been extensively studied in the literature and is known to be NP-hard.\\

\noindent\textbf{Proof of Theorem~\ref{thm3}.}
We are given a $d$-uniform hypergraph $H=(V,E)$ and a constant integer $c'$. 
We create a $d$-uniform hypergraph $\tilde{H}= (\tilde{V}, \tilde{E})$, where 

\begin{enumerate}
 
 \item
 $\tilde{V} = V \cup \{v'_0,v'_1,v'_2, \ldots, v'_{t(d-1)} \}$ where $t = \dbinom{|E|}{2}+1$. 
 
 \item
 $\tilde{E} = E \, \cup \, \{e_i \mid 1 \leq i \leq t \}$ where \scalebox{0.85}{$e_i= \{v'_0, v'_{(i-1)(d-1)+1}, v'_{(i-1)(d-1)+2}, \ldots, v'_{(i-1)(d-1)+(d-1)}\}$}.
\end{enumerate}

We prove that $\tilde{H}$ has a $d$-dimensional rectilinear drawing $D$ having at least $tc'$ crossing pairs of hyperedges if and only if there exists a partition of $V$ in two parts such that at least $c'$ hyperedges of $E$ contains at least one vertex from both  parts.\par

Let us assume that there exists a partition of $V$ comprising of two parts $V_1$ and $V_2$ such that (at least) $c'$ hyperedges of $E$ contain at least one vertex from both  parts. 
Let us denote these hyperedges as cut-hyperedges. 
We produce a drawing $D$ of $\tilde{H}$ having at least $tc'$ crossing pairs of hyperedges.\par

Let $h$ be a $(d-1)$-dimensional hyperplane. 
We place the points corresponding to the vertices in $V_1$ and the points corresponding to the vertices in $V_2$ in general position in $\mathbb{R}^d$ such that they lie on the different open half-spaces created by $h$.
The hyperedges in $E$ are drawn as the $(d-1)$-simplices spanned by the $d$ points corresponding to its vertices.
Note that each of the cut-hyperedges has a non-trivial intersection with $h$.
We then create the $t$ hyperedges $e_1,e_2, \ldots, e_t$.
Note that these $t$ hyperedges cannot form a crossing with each other since each of them contains a common vertex, i.e., $v'_0$.\par

We put the $d$ vertices $\{v'_0, v'_1,v'_2, \ldots, v'_{d-1}\}$ of $e_1$ on $h$ such that they are in general position with the rest of the points in $\mathbb{R}^d$ and the convex hull of these $d$ points crosses each of the cut-hyperedges.
Note that it is always possible to create such a placement of points since there are only a finite number of cut-hyperedges. 
Note that the position of the vertex $v'_0$ is fixed after the placement of the vertices of $e_1$. 
We then add the other $d-1$ vertices of $e_2$ very close to the $d-1$ vertices of $e_1$ such that they, along with the other vertices, maintain the general position and the $(d-1)$-simplex corresponding to the hyperedge $e_2$ crosses each of the cut-hyperedges. 
Following the same methodology, we keep on adding the vertices of each $e_i$ in a very close neighborhood of each other such that they do not violate the general position assumption and each $(d-1)$-simplex corresponding to each $e_i$ crosses the same number of cut-hyperedges. \par

Note that in this $d$-dimensional rectilinear drawing $D$ of $\tilde{H}$ (as depicted in Figure~\ref{fig:3}), each of the cut-hyperedges forms a crossing  with $e_i$ for each $1 \le i \le t$.
This implies that there exist at least $tc'$ crossing pairs of hyperedges in $D$.

\begin{figure}[!ht]
  \centering
  \includegraphics[scale = 0.5]{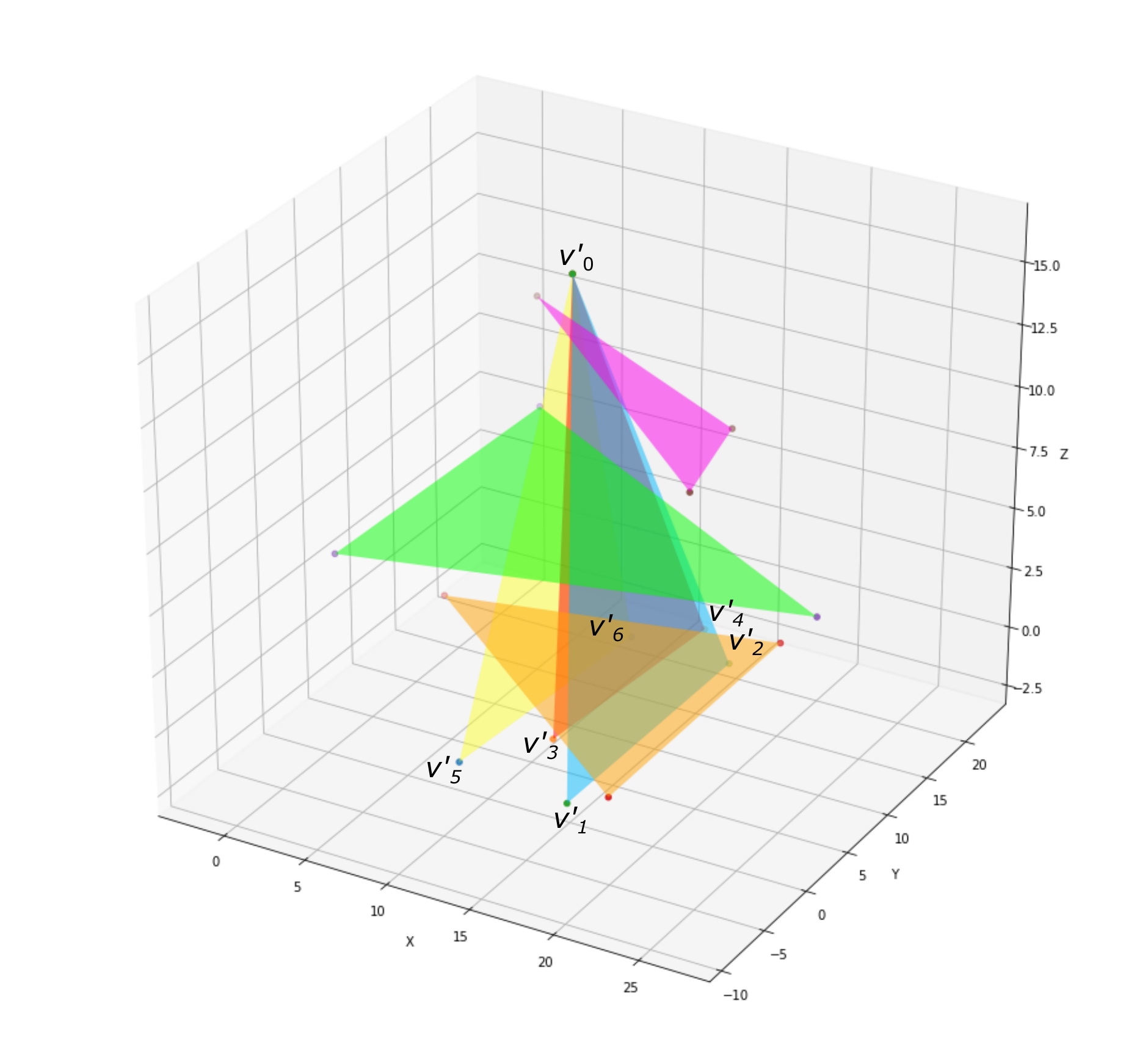}
  \caption{$3$-dimensional Rectilinear Drawing of a $3$-uniform hypergraph}
  \label{fig:3}
\end{figure}

On the other hand, let us assume that $\tilde{H}$ has a $d$-dimensional rectilinear drawing $D$ having at least $tc'$ crossing pairs of hyperedges. 
Suppose each $e_i$ crosses at most $(c'-1)$ hyperedges of $E$. 
Then, the maximum number of crossing pairs of hyperedges in $D$ is $(c'-1)t+\dbinom{|E|}{2} < (c'-1)t+t=c't$. \par

This implies that one of the $e_i$ must \emph{cross} at least $c'$ hyperedges of $E$.
Without loss of generality, suppose that $e_1$ crosses at least $c'$ hyperedges of $E$. 
Consider the hyperplane $h'$ spanned the $d$ vertices of $e_1$, i.e., the affine hull of the points $\{v'_0, v'_1,v'_2, \ldots, v'_{d-1}\}$. \par

The partition of $V$ created by $h'$ shows that there exists a partition of $V$ comprising of two parts $V_1$ and $V_2$ such that (at least) $c'$ hyperedges of $E$ contain at least one vertex from both the parts. \qed

\vspace{0.3cm}

\noindent\textbf{Proof of Theorem~\ref{lemrandom}.}
Pick a permutation uniformly at random of the vertices of $H$. 
Put the vertices on the $d$-dimensional moment curve in that order. 
We draw each hyperedge present in $E$ as a $(d-1)$-simplex formed by the corresponding vertices. 
Let this drawing of $H$ be denoted by $D'$. 
Let $X$ denote the number of crossing pairs of hyperedges in $D'$. 
Let $A'$ and $B'$ be a pair of vertex-disjoint hyperedges. 
Let $X_{A',B'}$ denote the indicator random variable, where 
$X_{A',B'}$ is $1$ if $A'$ and $B'$ form a crossing pairs of hyperedges, else it is set to $0$.\par

Note that the $2d$ vertices  can be placed on the $d$-dimensional moment curve in $c_d^m$ ways such that the $(d-1)$-simplex formed by the vertices of $A'$ and the $(d-1)$-simplex formed by the vertices of $B'$ form a crossing.
Note that we can permute $\{A', B'\}$ in two ways to obtain distinct ordered pairs, i.e., $(A', B')$, and $(B', A')$.
Also, note that vertices of $A'$ have $d!$ permutations among themselves. 
Similarly, vertices of $B'$ also have $d!$ permutations.\\

This implies that $Pr[X_{A',B'}=1]= \dfrac{2(d!)^2c_d^m}{(2d)!}=\dfrac{c_d^m}{\dbinom{2d-1}{d-1}}~.$
      
\noindent The expected number of crossing pairs of hyperedges in $D'$ is $\mathbb{E}(X)= \mathbb{E}(\sum_{\{A',B'\}}X_{A',B'})=\dfrac{c_d^m}{\dbinom{2d-1}{d-1}}\cdot F.$ Thus, there exists a random ordering of the vertices of $H$ over the $d$-dimensional moment curve which produces at least $\tilde{c_d}\cdot F$ crossing pairs of hyperedges. \qed

\noindent Note that $\tilde{c_d}$ is a constant.
Table~\ref{tab:1} contains the value of $\tilde{c_d}$ for $2 \leq d \leq 10$.

%

\begin{table}[!ht]

\begin{center}
\centering
\caption{Values of $\tilde{c_d}$}

\begin{tabular}{cc}
\hline
\multicolumn{1}{|c}{$d$} & \multicolumn{1}{c|}{ $\tilde{c_d}$}       \\ \hline
\multicolumn{1}{|c}{$ 2$} & \multicolumn{1}{c|}{1/3}    \\ 
\multicolumn{1}{|c}{$3$} & \multicolumn{1}{c|}{3/10}   \\ 
\multicolumn{1}{|c}{$4$} & \multicolumn{1}{c|}{13/35}   \\
\multicolumn{1}{|c}{$5$} & \multicolumn{1}{c|}{5/14}   \\ 
\multicolumn{1}{|c}{$ 6$} & \multicolumn{1}{c|}{181/462}  \\ 
\multicolumn{1}{|c}{$ 7$} & \multicolumn{1}{c|}{329/858}  \\ 
\multicolumn{1}{|c}{$ 8$} & \multicolumn{1}{c|}{521/1287} \\ 
\multicolumn{1}{|c}{$ 9$} & \multicolumn{1}{c|}{1941/4862} \\
\multicolumn{1}{|c}{$ 10$}& \multicolumn{1}{c|}{38251/92378}\\ \hline
\end{tabular}
\label{tab:1}
\end{center}
\end{table}


\section{Discussion}

In this paper, we have proved the conjecture of Anshu et al.~\cite{AGS}
for $d=4$ by proving that \emph{max}-
$\overline {cr}_4(K_{n}^4)= 13 \dbinom{n}{8}$.
The conjecture remains open for $d>4$.
Consider any $d$-dimensional neighborly polytope $P$ whose vertices are in general
position in $\mathbb{R}^d$. 
Since the vertices are in general position, $P$ is a $d$-dimensional simplicial neighborly polytope. The polytope
$P$ has the same $f$-vectors as the $d$-dimensional cyclic polytope~\cite{ZIE}.\par

Since $P$ is a $d$-dimensional simplicial  neighborly polytope, the number of crossing pairs of $\floor{d/2}$-simplex and $\ceil{d/2}$-simplex in $P$ is same as the number of crossing pairs of $\floor{d/2}$-simplex and $\ceil{d/2}$-simplex in the $d$-dimensional cyclic polytope.
Due to this similarity with the $d$-dimensional cyclic polytope, we believe that the number of crossing pairs of hyperedges is equal to $c_d^m$ if all the vertices of $K_{n}^d$ are placed  at the vertices of $P$.\par

We  conjecture that among all $d$-dimensional rectilinear drawings of
$K_{n}^d$, the number of crossing pairs of hyperedges gets maximized
if and only if all the vertices of $K_{n}^d$ are placed in general position in
$\mathbb{R}^d$ at the vertices of a $d$-dimensional neighborly polytope
(whose vertices are in general position).\par

Further, we want to ask a more general question in this area.
Consider a $d$-dimensional convex drawing of complete $d$-uniform
hypergraph having $2d$ vertices.
Note that  in a $d$-dimensional
convex drawing  of $K_{2d}^d$, its vertices are placed as vertices of a $d$-dimensional convex polytope.
As our results indicate, the $d$-dimensional convex polytopes with maximum number of
facets also maximize the number of crossing pairs of hyperedges. 
It is an interesting problem to find out the relation between the number of
crossing pairs of hyperedges in a $d$-dimensional convex drawing of
$K_{2d}^d$ and the number of facets of the corresponding polytope. 
Guy~\cite{GUY} noted that in a rectilinear drawing of a complete graph,
the number of crossing pairs of edges is minimum when the convex hull of
its vertices forms a triangle. 
Aichholzer et al.~\cite{AGOR} proved this claim rigorously using continuous
motion of the vertices. 
It is a nice problem to prove that the convex hull of the vertices of $K_n^d$
in a $d$-dimensional rectilinear drawing of it is a $d$-simplex if the number
of crossing pair of hyperedges is minimum. \par

Goodman and Pollack~\cite{GP} proved that the lower bound on the realizable order types of $n$ points in general position in $\mathbb{R}^d$ is roughly in the order of $n^{d^2n+O(n/ \log n)}$. 
This shows that it is hard to extend our proof of Theorem~\ref{thm1} for higher dimensions. It is interesting to provide a traditional proof for  Theorem~\ref{thm1}. 
Such a proof can also be useful for proving the general conjecture. \par

Theorem~\ref{lemrandom} shows that there is a randomized approximation
algorithm which in expectation provides a $\tilde{c_d}$ guarantee on
the maximum $d$-dimensional rectilinear crossing number problem. 
It is an interesting open problem to derandomize the algorithm.
For $d=2$, Bald et al.~\cite{BLD} derandomized the algorithm. 
Note that $\tilde{c_d}$ is a constant for a given $d$. 
It is easy to observe that $\tilde{c_d}$ is bounded by $1/2$ from above. 
It would be good to give a lower bound on $\tilde{c_d}$. 
Our guess is $\tilde{c_d} \geq 3/10$.


\section*{Acknowledgements}

Some parts of the Research were  conducted while Rahul Gangopadhyay were at IIIT-Delhi and Saint-Petersburg State University.\footnote{Some parts
of this work were presented at EuroCG 2020 and EuroCG 2022.}
Rahul Gangopadhyay was   supported by Ministry of Science and Higher Education of the Russian Federation, agreement no. 075-15-2022-287. Authors are deeply grateful to Dr. Gaiane Panina for the help in the proof of Theorem~\ref{thm0}. 
\section*{Statements and Declarations}
The Authors declare that there is no conflict of interest.


\begin{thebibliography}{}

\bibliographystyle{plain}


\bibitem{ALCHO1}
Aichholzer,O.:
Order Types for Small Point Sets.\\
http://www.ist.tugraz.at/staff/aichholzer/research/rp/triangulations/ordertypes/

\bibitem{ALCHO}
 Aichholzer, O., Aurenhammer, F., Krasser, H.:
Enumerating order types for small point sets with applications. 
Order 19, 265-281 (2002).

\bibitem{FML}
 Aichholzer, O.,  Duque, F.,  Fabila-Monroy, R.,  Hidalgo-Toscano, C., García-Quintero,  O. E.: 
 An ongoing project to improve the rectilinear and the pseudolinear crossing constants. 
Journal of Graph Algorithms and Applications 24, 421-432 (2020).
\bibitem{AGOR}
 Aichholzer, O.,  Garc{\'\i}a, J.,  Orden, D.,  Ramos, P.:
 New lower bounds for the number of ($ \leq k$)-edges and the rectilinear crossing number of $K_n$. 
Discrete and Computational Geometry 38, 1-14 (2007).

 

\bibitem{AA}
Akiyama, J.,  Alon, N.: 
Disjoint simplices and geometric hypergraphs.
Annals of the New York Academy of Sciences 555, 1-3 (1989).

\bibitem{AGS}
Anshu, A., Gangopadhyay, R., Shannigrahi, S., Vusirikala, S.: 
On the rectilinear crossing number of complete uniform hypergraphs.
Computational Geometry: Theory and Applications 61, 38-47 (2017).

\bibitem{Alb}
 \'Abrego, B. M.,  Cetina, M.,  Fern\'andez-Merchant, S.,
 Lea\~{n}os, J.,  Salazar, G.: On $(\le k)$-edges,
crossings, and halving lines of geometric drawings of $K_n$.
{Discrete and Computational Geometry}
48, 192-215 (2012).

\bibitem{AS}
Anshu, A.,  Shannigrahi, S.: 
A lower bound on the crossing number of uniform hypergraphs.
Discrete Applied Mathematics 209, 11-15 (2016).

\bibitem{ABRGO}
\'Abrego, B. M.,   Fern\'andez-Merchant, S.,  Salazar, G.:
The rectilinear crossing number of $K_n$: Closing in (or are we?). 
Thirty essays on geometric graph theory, 5-18 (2013). 
Springer, New York, NY.
 

\bibitem{BLD}
Bald, S.,  Johnson, M. P., Liu, O.: 
Approximating the maximum rectilinear crossing number. 
In Proceedings of International Computing and Combinatorics Conference. 
Springer, 455-467(2016).

\bibitem{BAR}
B{\'a}r{\'a}ny, I.:
A generalization of Carath{\'e}odory's theorem.
Discrete Mathematics 40, 141-152 (1982).

\bibitem{BR}
 Breen, M.:
Primitive Radon partitions for cyclic polytopes.
Israel Journal of Mathematics 15, 156-157 (1973).

\bibitem{BISZ}
 Bisztriczky, T., Soltan, V.:
Some Erd\H{o}s-Szekeres type results about points in space, 
Monatshefte f\"{u}r Mathematik 118, 33-40 (1994).

\bibitem{CHI}
Chimani, M.,  Felsner, S., Kobourov, S., 
 Ueckerdt, T.,   Valtr, P.,  Wolff, A.: 
On the Maximum Crossing Number.
Journal of Graph Algorithms and Applications 22, 67-87 (2018).

\bibitem{DP}
 Dey, T. K.,  Pach, J.:
Extremal problems for geometric hypergraphs.
Algorithms and Computation (Proc. ISAAC '96, Osaka; T. Asano et al., eds.),
Lecture Notes in Computer Science 1178, Springer-Verlag, 105-114 (1996).
Also in: Discrete and Computational Geometry 19, 473-484 (1998).

\bibitem{GL}
 Gale, D.:
Neighboring vertices on a convex polyhedron.
Linear inequalities and related system 38, 255-263 (1956).

\bibitem{GS}
Gangopadhyay, R.,  Shannigrahi, S.: 
$k$-Sets and Rectilinear Crossings in Complete Uniform Hypergraphs.
Computational Geometry: Theory and Applications 86, 101578 (2020).
\bibitem{GS1}
Gangopadhyay, R., Shannigrahi, S.:
Rectilinear Crossings in Complete Balanced $d$-Partite $d$-Uniform Hypergraphs. Graphs and Combinatorics 36, 905-911 (2020).

\bibitem{GP}
Goodman, J. E.,  Pollack, R.: 
Upper bounds for configurations and polytopes in $\mathbb{R}^d$. 
Discrete and Computational Geometry 1, 219-227 (1986).


\bibitem{GRU}
Gr{\"u}nbaum, B.:
Convex Polytopes. 
Springer, 2003.

\bibitem{GUY}
Guy, R. K.: 
Crossing numbers of graphs. 
Graph Theory and Applications, 111-124 (1972).
 
\bibitem{LOV}
Lov{\'a}sz, L.: 
Coverings and colorings of hypergraphs.
In Proceedings of the 4th Southeastern Conference on
Combinatorics, Graph Theory and Computing. 
Utilitas Mathematica Publishing, 3-12 (1973). 
 
\bibitem{JM}
Matou\v{s}ek, J.: 
Lectures in Discrete Geometry. 
Springer, 2002.

\bibitem{MM}
McMullen, P.:
The maximum numbers of faces of a convex polytope.
Mathematika 17, 179-184 (1970).


\bibitem{RIN}
Ringel, G.:
Extremal problems in the theory of graphs. 
In Proceedings of Theory of
Graphs and its Applications, 85-90 (1964).

\bibitem{SCH}
Schaefer, M.: 
The graph crossing number and its variants: A survey. 
The electronic journal of combinatorics 1000, 21-22 (2013).

\bibitem{VER}
Verbitsky, O.: 
On the obfuscation complexity of planar graphs. 
Theoretical Computer Science 396, 294-300 (2008).
\bibitem{WW}
Wagner, U.,  Welzl, E.: 
A continuous analogue of the upper bound theorem. 
Discrete and Computational Geometry 26, 205-219 (2001).

\bibitem{ZIE}
 Ziegler, G. M.: 
Lectures on Polytopes. 
Springer, 1995.


\end{thebibliography}
\end{document}